\documentclass[11pt,a4paper]{article}
\RequirePackage{amsmath}%
\RequirePackage{amsthm}%
\usepackage{amsfonts}%
\usepackage{amssymb}%
\usepackage{graphicx}%
\usepackage{float}%
\usepackage{MnSymbol}%
\usepackage{epstopdf}%
\usepackage{mathtools}
\usepackage{url}
\newtheorem{theorem}{Theorem}[section]
\newtheorem{lemma}[theorem]{Lemma}

\theoremstyle{definition}
\newtheorem{example}[theorem]{Example}
\newtheorem{remark}[theorem]{Remark}

\renewcommand{\Re}{\mathop{\mathrm{Re}}}
\renewcommand{\Im}{\mathop{\mathrm{Im}}}
\renewcommand{\i}{\mathrm{i}}
\renewcommand{\d}{\mathrm{d}}

\newcommand{\CC}{{\mathbb C}}

\newcommand{\capp}{\mathrm{cap}}
\newcommand{\C}{{\mathbb C}}
\newcommand{\R}{{\mathbb R}}
\newcommand{\N}{{\mathbb N}}
\newcommand{\cL}{\mathcal{L}}
\DeclareMathOperator{\diag}{diag}
\DeclareMathOperator{\dist}{dist}
\DeclarePairedDelimiter{\abs}{\lvert}{\rvert}
\DeclarePairedDelimiter{\cc}{[}{]}

\title{Computing the logarithmic capacity of compact sets having (infinitely)
many components with the Charge Simulation Method}
\author{J{\"o}rg Liesen\footnotemark[2] \and Mohamed M.S. Nasser\footnotemark[3]
\and Olivier S\`ete\footnotemark[4]}
\date{March 1, 2023}

\begin{document}
\maketitle

\renewcommand{\thefootnote}{\fnsymbol{footnote}}

\footnotetext[2]{Institute of Mathematics, Technische Universit\"{a}t
Berlin, Stra{\ss}e des 17. Juni 136, 10623 Berlin, Germany.
\texttt{liesen@math.tu-berlin.de}, ORCID: 0000-0002-3677-373X}

\footnotetext[3]{Department of Mathematics, Statistics, and Physics, Wichita 
State University, Wichita, KS 67260-0033, USA.
\texttt{mms.nasser@wichita.edu}, ORCID: 0000-0002-2561-0978}

\footnotetext[4]{Institute of Mathematics and Computer Science, Universit\"at
Greifswald, Walther-Rathenau-Stra\ss{}e 47, 17489 Greifswald, Germany.
\texttt{olivier.sete@uni-greifswald.de}, ORCID: 0000-0003-3107-3053}

\renewcommand{\thefootnote}{\arabic{footnote}}

\begin{abstract}
We apply the Charge Simulation Method (CSM) in order to compute the logarithmic
capacity of compact sets consisting of (infinitely) many ``small'' components.
This application allows to use just a single charge point for each 
component. The resulting method therefore is significantly more efficient 
than methods based on discretizations of the boundaries (for example, our own 
method presented in~\cite{LSN}), while maintaining a very high level of accuracy.  We study properties of the linear algebraic systems that arise in the CSM, and show how these systems can be solved efficiently using preconditioned iterative methods,
where the matrix-vector products are computed using the Fast Multipole Method.
We illustrate the use of the method on generalized Cantor sets and the Cantor
dust.
\end{abstract}


\medskip

\noindent{\bf AMS.}
65E05, 
30C85, 
31A15, 
65F10  

\pagestyle{myheadings}
\thispagestyle{plain}

\section{Introduction}
\label{sc:int}

As pointed out by Ransford and Rostand, the computation of the logarithmic
capacity of compact subsets of the complex plane is in general a ``notoriously
hard'' problem~\cite[p.~1499]{RanRos07}.
This is particularly true for non-connected sets consisting of many,
or even infinitely many components.
An important and frequently studied example in this context is given by the
classical Cantor middle third set and its generalizations, for which no analytic
formula for their logarithmic capacity is known.  Among the different approaches
to computing the logarithmic capacity of these fractal sets are the method
of Ransford and Rostand which uses linear programming~\cite{RanRos07}, a method
of Banjai, Embree and Trefethen based on Schwarz-Christoffel mappings (also
see~\cite{RanRos07}), the study of the spectral theory of orthogonal polynomials
by Kr{\"u}ger and Simon~\cite{KruSim15}, and our algorithm based on conformal
maps onto lemniscatic domains~\cite{LSN}.  Further numerical methods for
computing the logarithmic capacity of compact sets can be found, for example,
in
\cite{BaddooTrefethen2021,DijkstraHochstenbach2009, Ransford2010, Rostand1997}.

In this paper we derive and study an alternative method for
numerically approximating the logarithmic capacity of compact
sets consisting of very many ``small'' components, which uses
the \emph{Charge Simulation Method} (CSM)~\cite{Ama98,Ama10,Oka03a},
also known as the \emph{Method of Fundamental Solutions} \cite{aug15,wan19}.
In the CSM, the Green's function of the complement of the given compact set
is approximated by a linear combination of logarithmic potentials
which depend on so-called charge points inside each component of the set.
Solving the (dense) linear algebraic system for the coefficients of
the linear combination then yields the desired approximation of the
logarithmic capacity.

For the types of sets we consider, choosing one charge point in each component
is sufficient. A similar approach has been used in~\cite{KN}, where the CSM
was used to solve the harmonic image inpainting problem. The linear algebraic
system that needs to be solved therefore is significantly smaller
than in methods that are based on discretizing the boundaries of each component
(for example, our own Boundary Integral Equation (BIE) method presented
in~\cite{LSN}). We solve the linear
algebraic systems arising in the CSM iteratively with the
GMRES method~\cite{SaaSch86}. In order to speed
up the iteration we use the centrosymmetric structure of the
system matrices, the Fast Multipole Method~\cite{GR} for computing 
matrix-vector products, and a problem-adapted preconditioner. We apply the new 
method to generalized Cantor sets and the Cantor dust.  These sets consist of 
infinitely many components, and we compute approximations of their logarithmic 
capacities by extrapolating from computed logarithmic capacities of finite 
approximations that consist of many components.  Timing comparisons with our 
method from~\cite{LSN} show the computational efficiency, and comparisons of the
computed capacity values with other published results demonstrate the high
accuracy of the new method.

The paper is organized as follows.  In Section~\ref{sec:CSM} we present the
necessary background on the logarithmic capacity and the CSM.  In
Section~\ref{sec:Gen-Cantor} we study the approximation of the logarithmic
capacity for generalized Cantor sets, and in Section~\ref{sec:dust} we do the
same for the Cantor dust. The paper ends with concluding remarks
in Section~\ref{sec:Concl}.

\section{Logarithmic capacity and the CSM}\label{sec:CSM}

Let $E \subseteq \C$ be a compact set and denote $E^c := (\C \cup \{\infty\})
\setminus E$.  Suppose that the unbounded connected component $G$ of $E^c$ is 
regular in the sense that it possesses a \emph{Green's function} $g = g_E = 
g_G$.
Then the \emph{logarithmic capacity} of $E$, denoted by $\capp(E)$, satisfies
\begin{equation} \label{eqn:cap}
\capp(E) = \lim_{z \to \infty} \exp(\log \abs{z} - g(z));
\end{equation}
see, e.g., \cite[Theorem~5.2.1]{Ransford1995}, \cite[Eq.~(2.2)]{LSN}, or 
Szeg\H{o}'s original article~\cite{Szego1924}.
Recall that the Green's function with pole at infinity of $G$ is the real-valued
function such that
\begin{enumerate}
\item $g$ is harmonic in $G \setminus \{ \infty \}$ and $g(z) - \log \abs{z}$
is bounded in a neighborhood of infinity,
\item $g$ is continuous in $\overline{G} \setminus \{ \infty \}$ and vanishes
on $\partial G$.
\end{enumerate}
We assume that the boundary of the unbounded domain $G$ with $\infty \in G$
consists of a
finite number of Jordan curves, in which case the Green's function $g$ with pole
at infinity of $G$ exists~\cite[p.~41]{GarnettMarshall2005}.

In the CSM, a harmonic function is approximated by a
linear combination of fundamental solutions of the Laplace equation, which are
logarithmic potentials in the two-dimensional case~\cite{aug15,wan19}.
In our setting, the Green's function $g$ is approximated by a function of the 
form
\begin{equation} \label{eq:CSM_ansatz}
h(z) = c + \sum_{j=1}^N p_j \log|z-w_j|,  \quad z \in G,
\end{equation}
where $c, p_1, \ldots, p_N$ are undetermined real constants, and where $w_1,
\ldots, w_N \in \C \setminus \overline{G}$ are pairwise distinct.
The points $w_1, \ldots, w_N$ are called the \emph{charge points}, and the
coefficients $p_1, \ldots, p_N$ the \emph{charges}~\cite{Ama98}.
The function $h$ in~\eqref{eq:CSM_ansatz} is harmonic in $\C \setminus \{ w_1,
\ldots, w_N \}$.
Since $g$ behaves as $\log \abs{z}$ for $z \to \infty$, we require
\begin{equation} \label{eqn:sum_pj}
\sum_{j=1}^N p_j = 1.
\end{equation}
The coefficients $c, p_1, \ldots, p_N$ in~\eqref{eq:CSM_ansatz} are usually
determined from the boundary condition $g(z) = 0$ on $\partial G$,
by imposing the condition $h(z)=0$ in a finite set of collocation points on
$\partial G$; see, e.g.,~\cite{Ama98,
Ama10, aug15, Oka03a, wan19}.  The number of collocation points is usually at
least $N+1$, so that this procedure leads to a square or overdetermined linear
algebraic system~\cite[p.~12]{wan19}.
Below we will use a slightly different approach to obtain a linear algebraic
system for the coefficients.

In view of~\eqref{eqn:cap}, the approximation $h$ of $g$ yields
\begin{equation} \label{eqn:capE_approx}
\capp(E)
= \lim_{z \to \infty} \exp(\log \abs{z} - g(z))
\approx \lim_{z \to \infty} \exp(\log \abs{z} - h(z))
= e^{-c}.
\end{equation}
The following result gives an error bound for this approximation.

\begin{lemma}\label{lem:bd}
Let $E \subseteq \C$ be compact such that the unbounded connected component $G$
of $E^c$ is bounded by $m$ Jordan curves, and let $g$ be the Green's function
with pole at infinity of $G$.  Let $h$ be as in~\eqref{eq:CSM_ansatz} 
with~\eqref{eqn:sum_pj}, then
\begin{equation}
\abs{\capp(E) - e^{-c}}
\leq e^{-c} \Big( M + \frac{1}{2} M^2 e^M \Big), \label{eqn:estimate_capacity}
\end{equation}
where $M \coloneq \abs{\hat{c}-c} \leq \max_{\zeta \in \partial G}
\abs{h(\zeta)}$ and
$\hat{c} \coloneq \lim_{z \to \infty} (g(z) - \log \abs{z}) \in \R$.
\end{lemma}

\begin{proof}
As above, let $w_1\in\C\setminus \overline{G}$ be the first charge point.
The auxiliary functions
\begin{equation*}
u(z) = g(z) - \log \abs{z - w_1}, \quad
v(z) = h(z) - \log \abs{z - w_1},
\end{equation*}
are continuous in $\overline{G}$ and harmonic in $G$, including at infinity
with $u(\infty) = \hat{c}$ and $v(\infty) = c$.
Since $w_1 \notin \overline{G}$, the M\"obius transformation $\varphi(z) = 1/(z
- w_1)$ maps $G$ onto the bounded domain $R = \varphi(G)$, whose boundary
consists of $m$
Jordan curves.
The functions $u \circ \varphi^{-1}$ and $v \circ \varphi^{-1}$ are harmonic in
$R$ and continuous in $\overline{R}$.
Then, by the maximum principle for harmonic functions on bounded domains (see,
e.g.,~\cite[Theorem~1.1.8]{Ransford1995}),
\begin{equation*}
\abs{(u \circ \varphi^{-1})(w) - (v \circ \varphi^{-1})(w)}
\leq \max_{\omega \in \partial R} \,\abs{(u \circ \varphi^{-1})(\omega) - (v
\circ \varphi^{-1})(\omega)}, \quad w \in \overline{R}.
\end{equation*}
This yields
\begin{equation*}
\abs{u(z) - v(z)}
\leq \max_{\zeta \in \partial G} \abs{u(\zeta) - v(\zeta)}
= \max_{\zeta \in \partial G} \,\abs{g(\zeta) - h(\zeta)}
= \max_{\zeta \in \partial G} \,\abs{h(\zeta)}, \quad z \in \overline{G}.
\end{equation*}
The last equality holds since $g$ vanishes on $\partial G$.
Taking the limit $z\to\infty$ we obtain
$\abs{\hat{c} - c} \leq \max_{\zeta \in \partial G}\, \abs{h(\zeta)}$.
The estimate~\eqref{eqn:estimate_capacity} follows from $\capp(E) = e^{-
\hat{c}}$ and Taylor's formula.
\end{proof}

On the right hand side of~\eqref{eqn:estimate_capacity} we can replace
the (in general unknown) quantity $M$ by the upper bound given by the maximum
of the (known) function $h$ on $\partial G$.  In this way we obtain a
computable upper bound on the error. We will give an example in
Section~\ref{sec:two-disks}.

\begin{remark}
We give another interpretation of the approximation $e^{-c} \approx \capp(E)$
in the case $p_1, \ldots, p_N > 0$. Then $h(z)>0$ is equivalent to
$\prod_{j=1}^N \abs{z - w_j}^{p_j} > e^{-c}$. Since the real numbers
$p_1, \ldots, p_N>0$ are not necessarily rational, the set
$\cL = \{ z \in \C \cup \{ \infty \} : h(z) > 0 \}$ is the exterior of
a generalized lemniscate. This is an unbounded domain (with $\infty \in \cL$),
and $h$ is the Green's function with pole at infinity of $\cL$. In particular,
$e^{-c}$ is the logarithmic capacity of the compact set $\cL^c$.
If $\cL$ has connectivity $N$, it is a \emph{lemniscatic domain};
see~\cite{SeteLiesen2016, NasserLiesenSete2016} and references therein as well
as Walsh's original article~\cite{Walsh1956}.
Otherwise, the connectivity of $\cL$ is lower, and $\cL$ is the exterior of a
level curve of the Green's function of a lemniscatic domain.
Thus, in the method we propose here, the original domain $G = E^c$
is approximated by the domain $\cL$, and the method returns $\capp(\cL^c)$
as approximation of $\capp(G^c) = \capp(E)$.

If not all charges are positive, then the set $\{ z \in \C \cup \{ \infty \} :
h(z) > 0 \}$ is a canonical domain of the more general form
in~\cite[Theorem~3]{Walsh1956}.
\end{remark}

In the following two sections, we will apply the CSM in order to compute 
(approximations of) the logarithmic capacity of generalized Cantor sets 
(Section~\ref{sec:Gen-Cantor}) and the Cantor dust (Section~\ref{sec:dust}).

\section{Generalized Cantor sets}
\label{sec:Gen-Cantor}

In this section, we start with setting up the CSM for the generalized Cantor 
sets.  Next, we give an analytic example (illustrating this approach and 
Lemma~\ref{lem:bd}), and study
the structure and properties of the matrices arising in the CSM.  We then
show how the resulting linear algebraic systems can be solved iteratively,
and finally we present the results of numerical computations of the
logarithmic capacity of generalized Cantor sets.

\subsection{Setting up the CSM}
\label{sec:CSM_Cantor_set}

Fix some $q\in(0,1/2)$, let $E_0 \coloneq [0,1]$, and define recursively
\begin{equation}\label{eq:Ek}
E_k:=qE_{k-1}\cup\left( q E_{k-1}+ 1-q \right), \quad k\ge 1.
\end{equation}
Thus, the set $E_k$ is obtained by removing the middle $1-2q$ from each
interval of the set $E_{k-1}$; see Figure~\ref{fig:set} for $E_1$ (left) and
$E_2$ (right) corresponding to $q=1/4$.
Then the generalized Cantor set $E=E(q)$ is defined as
\begin{equation}\label{eq:E}
E \coloneq \bigcap_{k=0}^{\infty} E_k.
\end{equation}
For $q=1/3$ we obtain the classical Cantor middle third set. The limiting cases
are $E=\{0,1\}$ for $q=0$ and $E=[0,1]$ for $q=1/2$.

\begin{figure}
{\centering
\includegraphics[width=0.45\linewidth]{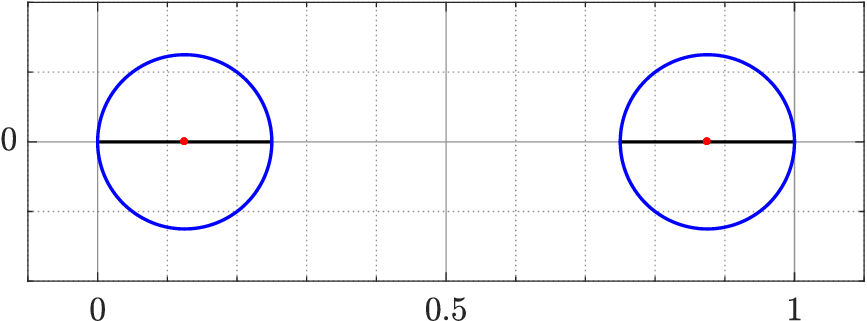}\hfill
\includegraphics[width=0.45\linewidth]{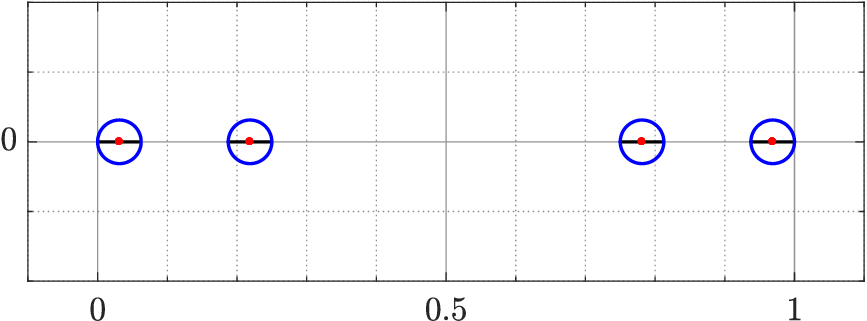}

}
\caption{The sets $E_1, D_1$ (left) and $E_2, D_2$ (right) for $q=1/4$.}
\label{fig:set}
\end{figure}

The set $E_k$ consists of $m = 2^k$ disjoint intervals $I_{k,j}$, $j = 1, 2, 
\ldots, m$, numbered from left to right.  The intervals have same length
\begin{equation}
\abs{I_{k,1}} = \abs{I_{k,2}} = \ldots = \abs{I_{k,m}} = q^k.
\end{equation}
Denote the midpoint of $I_{k,j}$ by $w_{k,j}$.  Let
\begin{equation}\label{eq:rk}
r_k = \frac{1}{2} q^k = w_{k,1},
\end{equation}
then
\begin{equation} \label{eqn:order_wk}
0 < r_k = w_{k,1} < w_{k,2} < \cdots < w_{k,m} = 1-r_k<1.
\end{equation}

Instead of the sets $E_k$ consisting of closed intervals we will use sets $D_k$
consisting of closed disks in the CSM. More precisely, let
\begin{equation*}
D_0 \coloneq \left\{z \in \C
\,:\, \abs*{z - \frac{1}{2}} \le\frac{1}{2} \right\},
\end{equation*}
and define recursively
\begin{equation*}
D_k:=qD_{k-1}\cup\left(qD_{k-1}+1-q\right), \quad k\ge 1.
\end{equation*}
Then $D_k$ consists of $m=2^k$ disjoint disks $D_{k,j}$ with the centers
$w_{k,j}$, $j=1,2,\ldots,m$, i.e.,
\begin{equation*}
D_k=\bigcup_{j=1}^m D_{k,j};
\end{equation*}
see Figure~\ref{fig:set} for $D_1$ (left) and $D_2$ (right) corresponding to
$q=1/4$.  All of the $m$ disks have the same radius $r_k$ from~\eqref{eq:rk}.
Since $E_0 = D_0 \cap \R$, we obtain $E_k = D_k \cap \R$ for all $k \geq 1$ by
induction.  Moreover, we have the following result.

\begin{theorem} \label{thm:E}
In the notation established above,
\[
\bigcap_{k=0}^{\infty} D_k = E
\quad \text{and} \quad
\capp(E) = \lim_{k \to \infty} \capp(D_k).
\]
\end{theorem}

\begin{proof}
Clearly, $E \subseteq D \coloneq \bigcap_{k=0}^\infty D_k$.  Next, we show that
if $z \in \C \setminus E$, then $z \notin D$.  If $z \in \R \setminus E$, then
there exists $k \in \N$ such that
$z \notin E_k = D_k \cap \R \supseteq D \cap \R$, hence $z \notin D$.
If $z \in \C \setminus \R$, i.e., $\abs{\Im(z)} > 0$, then there exists $k \in
\N$ such that $0 < r_k < \abs{\Im(z)}$, which shows that $z \notin D_k
\supseteq D$.  Together we obtain $D = E$.
Finally, $\capp(E) = \lim_{k \to \infty} \capp(D_k)$
by~\cite[Theorem~5.1.3]{Ransford1995}, since $D_0 \supseteq D_1 \supseteq D_2
\supseteq \ldots$ are compact and $E = \bigcap_{k=0}^\infty D_k$.
\end{proof}

Our overall strategy for computing an approximation of $\capp(E)$, where the
set $E$ consists of infinitely many components, is to first approximate
$\capp(D_k)$ for reasonably many and large values of~$k$ with the CSM.  Then we
extrapolate from the computed approximations of $\capp(D_k)$ to obtain an
approximation of $\lim_{k \to \infty} \capp(D_k) = \capp(E)$.

The set $G_k \coloneq D_k^c = (\CC\cup\{\infty\}) \setminus D_k$ is an unbounded
multiply connected domain of connectivity $m$ with
\begin{equation*}
\partial G_k = C_{k,1}\cup\cdots\cup C_{k,m},
\end{equation*}
where $C_{k,j} \coloneq \partial D_{k,j}$ is the circle with center $w_{k,j}$ 
and radius $r_k$ for $j=1,2,\ldots,m$.

As described in Section~\ref{sec:CSM}, we approximate $g_k$, the Green's 
function with pole at infinity of $G_k$, with the CSM by
\begin{equation} \label{eqn:approx_Green_function}
h_k(z) = c_k + \sum_{\ell = 1}^m p_{k,\ell} \log \abs{z - w_{k, \ell}}, \quad
z \in G_k,
\end{equation}
with
\begin{equation} \label{eq:sum-p}
\sum_{\ell=1}^m p_{k,\ell} = 1.
\end{equation}
In the CSM for unbounded multiply connected domains, we usually choose many
charge points inside each boundary component $C_{k,j}$,
$j=1,2,\ldots,m$~\cite{Ama10,Oka03a}.  However, because $r_k$ is very small for
large $m=2^k$, we choose only one point inside $C_{k,j}$, which is its center
$w_{k,j}$.

\begin{remark}\label{rmk:savings}
The fact that we use just \emph{a single charge point for each boundary
component} $C_{k,j}$ is an essential difference to other discretization-based
methods for computing the logarithmic capacity of sets consisting of many
components.  For example, our own BIE method presented in~\cite{LSN} ``opens
up''
the intervals of $E_k$ to obtain a compact set of the same capacity, but
bounded by smooth Jordan curves.  The computation of the logarithmic capacity
then is based on discretizing the $m=2^k$ boundary curves using $n$ points on
each of them.  Consequently, for the same
$k$ the linear algebraic systems in the method presented here are smaller by a
factor of~$n$ compared to those in~\cite{LSN}.  This is one of the main reasons
for the very significant computational savings that we obtain with the new
method; see Section~\ref{sec:NumExGenC}.
Using only a single charge point for each boundary component can be
justified by the fact that the boundary components become very small when $k$
increases. This is illustrated in~\cite{KN} where, as mentioned in the
Introduction, a similar approach has been used. (In terms of the
current paper, the derivation of the system matrix in~\cite{KN} assumes that
the value of $r_k$ in the equation~\eqref{eqn:bd_cond} is negligible in
comparison with $|w_{k,j}-w_{k,\ell}|$ for $j\ne\ell$.)
\end{remark}

The Green's function $g_k$ of $G_k$ vanishes on the boundary $\partial G_k$.  
This does not hold for the approximation $h_k$.  Instead, we require that $h_k$ 
has zero mean on each circle $C_{k,j}$, i.e.,
\begin{equation} \label{eqn:hk_zero_in_the_mean}
\frac{1}{2 \pi \i} \int_0^{2 \pi} h_k(\eta_{k,j}(t)) \, \d t = 0, \quad j = 1, 
\ldots, m,
\end{equation}
with the parametrization $\eta_{k,j}(t)=w_{k,j}+r_k e^{\i t}$, $0 \le t \le 2 
\pi$.
On the circle $C_{k,j}$,
\begin{align}
h_k(\eta_{k,j}(t))
&= c_k + \sum_{\ell = 1}^m p_{k,\ell} \log \abs{w_{k,j} + r_k e^{\i t} - w_{k,
\ell}}, \notag \\
&= c_k + p_{k,j} \log r_k + \sum_{{\begin{subarray}{l}\ell=1 \\ \ell\ne
j\end{subarray}}}^m p_{k,\ell} \log \abs{w_{k,j} + r_k e^{\i t} - w_{k,
\ell}}, \quad j = 1, \ldots, m,
\label{eqn:bd_cond}
\end{align}
for all $t \in \cc{0, 2 \pi}$.  For $\ell \neq j$, the function $\log \abs{z -
w_{k,\ell}}$ is harmonic in the disk $D_{k,j}  = \{ z \in \C : \abs{z-w_{k,j}}
\leq r_k \}$, hence
\begin{equation*}
\frac{1}{2\pi}\int_{0}^{2\pi}\log|w_{k,j}+r_ke^{ \i t}-w_{k,\ell}| \, \d t
= \log \abs{w_{k,j} - w_{k, \ell}}
\end{equation*}
by the mean value property of harmonic functions; see,
e.g.,~\cite[Theorem~4.6.7]{Wegert2012} or~\cite[Theorem~1.1.6]{Ransford1995}.
Thus, integrating~\eqref{eqn:bd_cond} and 
requiring~\eqref{eqn:hk_zero_in_the_mean} yields the linear algebraic system
\begin{equation}\label{eq:sys-5}
c_k+p_{k,j}\log r_k+\sum_{{\begin{subarray}{l}\ell=1 \\ \ell\ne
j\end{subarray}}}^{m}p_{k,\ell}\log|w_{k,j}-w_{k,\ell}|=0, \quad j=1,2,\ldots,m.
\end{equation}

We will now consider~\eqref{eq:sys-5} for a fixed size $m=2^k$, and therefore
we will drop the index~$k$ in the following for simplicity.  We
write~\eqref{eq:sys-5} in the form
\begin{equation}\label{eq:sys-5a}
A\mathbf{p}=c\,\mathbf{e},
\end{equation}
where $\mathbf{p}=[p_{1},\dots,p_{m}]^T$, $\mathbf{e}=[1,\dots,1]^T\in\R^m$, and
\begin{equation}\label{eq:A}
A = -
\begin{bmatrix}
\log r &\log|w_{1}-w_{2}| &\cdots &\log|w_{1}-w_{m-1}| &\log|w_{1}-w_{m}|  \\
\log|w_{2}-w_{1}| &\log r &\cdots &\log|w_{2}-w_{m-1}| &\log|w_{2}-w_{m}| \\
\vdots &\vdots                   &\ddots &\vdots              &\vdots \\
\log|w_{m-1}-w_{1}| &\log|w_{m-1}-w_2|  &\cdots &\log r  &\log|w_{m-1}-w_m| \\
\log|w_{m}-w_{1}| &\log|w_{m}-w_{2}|  &\cdots &\log|w_{m}-w_{m-1}| &\log r \\
\end{bmatrix}.
\end{equation}
Note that $A \in \R^{m,m}$ and that $m = 2^k$ is even.

\begin{theorem} \label{thm:A}
For $k \geq 1$, the entries $a_{ij}$ of $A$ satisfy
\begin{equation} \label{eqn:bound_aij}
\log \frac{1}{1 - 2r} \leq a_{ij} \leq \log \frac{1}{r}, \quad 1 \leq i, j \leq
m,
\end{equation}
and they decay away from the diagonal in each row and column, i.e.,
\begin{align}
a_{i,1} < \ldots < a_{i, i-1} < a_{i,i}, \quad
a_{i,i} > a_{i, i+1} > \ldots > a_{i,m}, \label{eqn:decay_row_A} \\
a_{1,j} < \ldots < a_{j-1, j} < a_{j,j}, \quad
a_{j,j} > a_{j+1, j} > \ldots > a_{m,j}, \label{eqn:decay_col_A}
\end{align}
for $1 \leq i, j \leq m$.
\end{theorem}

\begin{proof}
Since $0 < r = w_1 < w_2 < \ldots < w_m$, we have that
\begin{equation*}
\abs{w_i - w_1} > \abs{w_i - w_2} > \ldots > \abs{w_i - w_{i-1}} > r,
\quad
r < \abs{w_i - w_{i+1}} < \ldots < \abs{w_i - w_m},
\end{equation*}
which is equivalent to~\eqref{eqn:decay_row_A}.
Since $A$ is symmetric, \eqref{eqn:decay_row_A} is equivalent
to~\eqref{eqn:decay_col_A}.

Let $1 \leq i, j \leq m$.
The upper bound in~\eqref{eqn:bound_aij} follows from $a_{ij} \leq a_{ii} =
\log(1/r)$.
For the lower bound, notice that $\abs{w_i - w_j} \leq 1 - 2r$
by~\eqref{eqn:order_wk}, and that $1 - 2r > 0$ if and only if $k \geq 1$.
We then obtain $a_{ij} \geq \log \frac{1}{1 - 2r}$ for $i \neq j$, which also
holds for $a_{ii}$, since $a_{ii}$ is the largest element in the row.
\end{proof}

We will continue under the assumption that $A$ is nonsingular, which will be
justified below; see Section~\ref{sect:system} as well as Figure~\ref{fig:cond} 
and the corresponding discussion.
Note that~\eqref{eq:sum-p} can be written as $\mathbf{e}^T\mathbf{p}=1$, and
therefore multiplying~\eqref{eq:sys-5a} from the left with $\mathbf{e}^TA^{-1}$
yields
$1 = \mathbf{e}^T\mathbf{p} = c\,\mathbf{e}^T A^{-1} \mathbf{e}$ or, 
equivalently,
\begin{equation}\label{eq:c-A}
c=\frac{1}{\mathbf{e}^TA^{-1}\mathbf{e}}.
\end{equation}
Thus, in order to compute $e^{-c}$, which is our approximation of $\capp(E)$
(see \eqref{eqn:capE_approx}), we need to compute (or accurately estimate) the
quantity $\mathbf{e}^TA^{-1}\mathbf{e}$, preferably without explicitly
inverting the full matrix $A$.
One option is to numerically solve the linear algebraic system
\begin{equation}\label{eqn:Ax-e}
A\mathbf{x}=\mathbf{e},
\end{equation}
and then to compute $c = 1/(\mathbf{e}^T \mathbf{x})$.

\subsection{An analytic example: Two disks}\label{sec:two-disks}

In this section, we study the accuracy of the CSM approximation and the
tightness of the bounds in Lemma~\ref{lem:bd} on a simple example where an
exact formula for the logarithmic capacity is known.  We consider the set $D_1$
consisting of the union of the two disks with the centers
\[
w_1=\frac{1}{6}, \quad w_2=\frac{5}{6},
\]
and the radius $r$, where $0<r\le w_1=1/6$.
Denote by
\begin{equation*}
G=\{ z : |z-1/6|>r \text{ and } |z-5/6|>r \}
\end{equation*}
the complement of $D_1$ in the extended complex plane. Then $z \mapsto
z-\frac{1}{2}$ maps $G$ onto the domain
$\widehat G=\{z\,:\,|z+1/3|>r \text{ and } |z-1/3|>r\}$, and it follows
from~\cite[Theorem~5.2.3]{Ransford1995} and~\cite[Theorem~4.2]{SeteLiesen2016}
(see also~\cite[Example~4.6]{LSN}) that
\begin{equation}\label{eq:cap-D1}
\capp(D_1)=\capp(\widehat G^c)
=\frac{2K}{\pi}\sqrt{(1/3)^2-r^2}\sqrt{2L(1+L^2)}
=\frac{2K}{3\pi}\sqrt{1-9r^2}\sqrt{2L(1+L^2)},
\end{equation}
where
\[
K = K(L^2) = \int_0^1 \frac{1}{\sqrt{(1-t^2) (1-L^4 t^2)}} \, \d t, \quad
L = 2 \rho \prod_{k=1}^\infty \left( \frac{1 + \rho^{8k}}{1 + \rho^{8k-4}}
\right)^2,
\]
and
\[
\rho = \frac{\sqrt{1/3 + r} - \sqrt{1/3 - r}}{\sqrt{1/3 + r} + \sqrt{1/3 - r}}
= \frac{\sqrt{1 + 3r} - \sqrt{1 - 3r}}{\sqrt{1 + 3r} + \sqrt{1 - 3r}}
=\frac{3r}{1 + \sqrt{1 - 9r^2}}.
\]
Thus, for $r\ll 1/6$ we have
\[
\rho\approx 3r/2, \quad L\approx 2\rho\approx 3r\ll 1,
\]
hence
\[
K = K(L^2) \approx K(0)=\frac{\pi}{2},
\]
and now~\eqref{eq:cap-D1} implies that
\begin{equation}\label{eq:cap-ap-ex}
\capp(D_1) \approx
\frac{1}{3}\times 1\times \sqrt{2L} \approx
\frac{1}{3} \sqrt{6r} =\sqrt{\frac{2r}{3}} \quad\mbox{(for $r\ll 1/6$).}
\end{equation}

For the CSM we set up the linear algebraic system \eqref{eqn:Ax-e} with the
system matrix
\begin{equation*}
A= -
\begin{bmatrix}
\log r            &\log|w_1-w_2|\\
\log|w_2-w_1|     &\log r           \\
\end{bmatrix}
 =
\begin{bmatrix}
-\log r            &-\log\frac{2}{3}\\
-\log\frac{2}{3} &-\log r           \\
\end{bmatrix}
\end{equation*}
from~\eqref{eq:A}, so that
\begin{equation*}
A^{-1} = \frac{1}{(\log r)^2- (\log \frac{2}{3})^2}
\begin{bmatrix}
-\log r & \log \frac{2}{3}\\
\log \frac{2}{3} &- \log r \\
\end{bmatrix},
\end{equation*}
and hence (see~\eqref{eq:c-A})
\begin{equation}\label{eq:cap-c}
c=-\frac{(\log r)^2- (\log\frac{2}{3})^2}{2(\log r - \log \frac{2}{3})}=
-\frac{1}{2}\left(\log r + \log\frac{2}{3}\right)
= - \frac{1}{2} \log \frac{2 r}{3}.
\end{equation}
Consequently, the CSM estimate is
\begin{equation}\label{eq:cap-ap-ap}
\capp(D_1)\approx\exp\left(-c\right)=\exp\left( \frac{1}{2} \log \frac{2r}{3}
\right)
=\sqrt{\frac{2r}{3}}.
\end{equation}
Comparing~\eqref{eq:cap-ap-ex} and~\eqref{eq:cap-ap-ap} shows that the CSM
gives very accurate results for $r \ll 1/6$.  This conclusion is illustrated in
Figure~\ref{fig:twodisks}, where the (blue) solid line shows that the error
between the exact capacity~\eqref{eq:cap-D1} and the CSM
estimate~\eqref{eq:cap-ap-ap}
is $2.87\times10^{-18}$ for $r=10^{-7}$, but increases to $0.01$ for $r=1/6$.

We next consider the bound~\eqref{eqn:estimate_capacity} from
Lemma~\ref{lem:bd}, i.e., the inequality
\begin{equation}\label{eq:bd-again}
\abs{\capp(D_1) - e^{-c}}
\leq e^{-c} \Big( M + \frac{1}{2} M^2 e^M \Big),
\end{equation}
where $M = \abs{\hat{c} - c} \leq \max_{\zeta \in \partial G}
\abs{h(\zeta)}=\max_{\zeta \in \partial D_1} \abs{h(\zeta)}$, and
\[
h(z) = c + p_1 \log|z-w_1| + p_2\log|z-w_2|,  \quad z \in G.
\]
Here, $c$ is given by~\eqref{eq:cap-c}, $\capp(D_1)$ is given
by~\eqref{eq:cap-D1}, $\hat{c} = - \log(\capp(D_1))$, and $[p_1,p_2]^T$ is the
solution of the linear algebraic system~\eqref{eq:sys-5a}, i.e.,
\[
\begin{bmatrix}
-\log r & -\log\frac{2}{3} \\
-\log\frac{2}{3} & - \log r \\
\end{bmatrix}
\begin{bmatrix} p_1 \\ p_2 \end{bmatrix}
= \begin{bmatrix} c \\ c \end{bmatrix}.
\]
Thus, in view of~\eqref{eq:cap-c}, we have $p_1=p_2=1/2$, and hence
\[
h(z) = \frac{1}{2}\log \frac{3}{2r} + \frac{1}{2} \log \abs*{z-\frac{1}{6}} +
\frac{1}{2} \log \abs*{z-\frac{5}{6}},
\]
which yields
\[
M \leq \max_{\zeta \in \partial D_1} \abs{h(\zeta)}
= \frac{1}{2}\log
\frac{3}{2}+\frac{1}{2}\log\left(r+\frac{2}{3}\right)
= \frac{1}{2}\log\left(\frac{3}{2}r+1\right).
\]
Using this upper bound on $M$ in \eqref{eq:bd-again} we obtain
\begin{equation}\label{eq:up-hM}
\abs{\capp(D_1) - e^{-c}} 
\leq
\sqrt{\frac{r}{6}} \log\left(\frac{3}{2}r+1\right)\left( 1 + \frac{1}{4}
\log\left(\frac{3}{2}r+1\right) \sqrt{\frac{3}{2}r+1} \right) =: \hat{M}.
\end{equation}
The values of the computable upper bound $\hat{M}$ on the absolute error are
shown by the dotted line in Figure~\ref{fig:twodisks}. We observe that for
larger values of $r$ the bound becomes quite tight.
Figure~\ref{fig:twodisks} also shows the error bound $e^{-c} (M + \frac{1}{2}
M^2 e^M)$, which we can compute in this example since $\capp(D_1)$ is known
analytically. Clearly, the bound is very tight in this example.

\begin{figure}
\centerline{
\scalebox{0.55}{\includegraphics[trim=0cm 0cm 0cm 0cm,clip]
{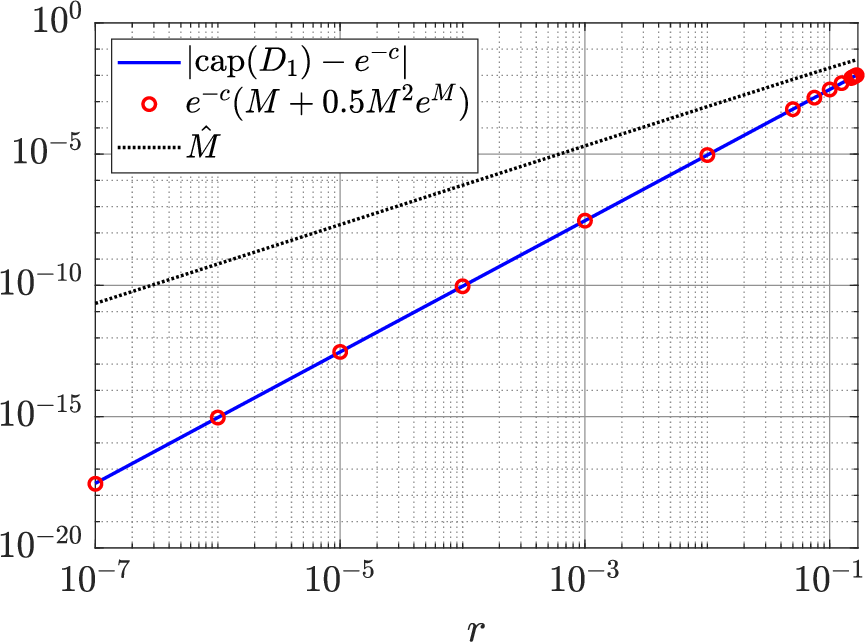}}
}
\caption{The absolute error $\abs{\capp(D_1) - e^{-c}}$ between the exact
capacity $\capp(D_1)$ given by~\eqref{eq:cap-D1} and the CSM estimate $e^{-c}$
in~\eqref{eq:cap-ap-ap}, the error bound from Lemma~\ref{lem:bd},
and the upper bound $\hat M$ in~\eqref{eq:up-hM}.}
\label{fig:twodisks}
\end{figure}

\subsection{Structure and properties of the system matrices}
\label{sect:system}

In Section~\ref{sec:two-disks} we have considered the case of just two disks,
and we were able to invert the $(2\times 2)$-matrix of the linear algebraic
system \eqref{eqn:Ax-e} explicitly.  For approximating the logarithmic capacity
of the generalized Cantor sets by extrapolating from the values $\capp(D_k)$
for reasonably large~$k$, we will have
to solve much larger linear algebraic systems, and thus we need to apply
more sophisticated techniques for the numerical solution of \eqref{eqn:Ax-e}.

Our first observation in this direction is that the matrix $A$ in~\eqref{eq:A}
is symmetric as well as centrosymmetric, which means that
\begin{equation*}
J_mAJ_m=A,\quad\mbox{where}\quad J_m=\begin{bmatrix} & & 1\\ & \udots & \\ 1 &
& \end{bmatrix}\in\mathbb{R}^{m\times m}.
\end{equation*}
Because of the centrosymmetry, the matrix~$A$ can be block-diagonalized with an
orthogonal transformation at no additional cost; see~\cite{FasIkr03}.  If we
partition
\begin{equation*}
A = \begin{bmatrix} A_{11} & A_{12}\\ A_{21} & A_{22} \end{bmatrix}, 
\quad\mbox{where}\quad
A_{11},A_{22}\in \mathbb{R}^{\frac{m}{2}\times \frac{m}{2}},
\end{equation*}
then
\begin{equation*}
A=Q\begin{bmatrix} B & 0\\ 0 & C\end{bmatrix}Q^T,\quad\mbox{where}\quad
Q=\frac{1}{\sqrt{2}}\begin{bmatrix}I_{m/2} & I_{m/2}\\ J_{m/2} &
-J_{m/2}\end{bmatrix} \quad \mbox{with}\quad Q^TQ=QQ^T=I_m,
\end{equation*}
and
\begin{equation} \label{eq:B}
B=A_{11}+A_{12}J_{m/2},\quad C=A_{11}-A_{12}J_{m/2}.
\end{equation}
Since $A$ is symmetric, the matrices $B,C\in\mathbb{R}^{\frac{m}{2}\times
\frac{m}{2}}$ are also symmetric.  If we partition
\begin{equation*}
\mathbf{x}=\begin{bmatrix}\mathbf{x}_1\\\mathbf{x}_2\end{bmatrix},
\end{equation*}
and use the orthogonal decomposition of $A$, then \eqref{eqn:Ax-e} can be
transformed via a left-multiplication with $Q^T$ into the equivalent system
\begin{equation*}
\begin{bmatrix} B  & 0 \\ 0 & C \end{bmatrix}
\begin{bmatrix}\mathbf{x}_1+J_{m/2}\mathbf{x}_2\\
\mathbf{x}_1-J_{m/2}\mathbf{x}_2 \end{bmatrix}=
\begin{bmatrix}2 {\mathbf{e}}\\ \mathbf{0}\end{bmatrix},
\end{equation*}
where $\mathbf{e}=[1,\dots,1]^T\in\R^{m/2}$.
Since we assume that $A$ and thus $C$ is nonsingular, the second block row
implies that $\mathbf{x}_1=J_{m/2}\mathbf{x}_2$, and hence it remains to solve
the system
\begin{equation}\label{eqn:B1-system}
B \mathbf{y}=\mathbf{e},
\end{equation}
where $\mathbf{y} = \mathbf{x}_1$, to obtain $c = 1/(2 \mathbf{e}^T 
\mathbf{y})$.  In finite precision, we obtain a computed approximate solution 
$\widetilde{\mathbf{y}} \approx \mathbf{y}$ of~\eqref{eqn:B1-system}, which 
leads to
\begin{equation*}
c \approx \frac{1}{2\mathbf{e}^T\widetilde{\mathbf{y}}}.
\end{equation*}

Figure~\ref{fig:cond} shows the $2$-norm condition numbers computed in
MATLAB\footnote{All computations in this paper are performed in MATLAB R2017a
on an ASUS Laptop with Intel Core i7-8750H CPU @ 2.20GHz, 2208 Mhz, 6 Cores,
12 Logical Processors and 16 GB RAM.} of
$A$ and $B$ for $q=1/3$ (i.e., the classical Cantor middle third set) as
functions of $m = 2^k$.  We observe that the condition
numbers grow linearly in $m$.
A similar behavior of the condition numbers
can be seen in the Cantor dust example (see Figure~\ref{fig:condd}), and has
been observed in~\cite[Figure~6]{HelWad05}, where the matrix $F$
in~\cite[Eq.~(36)]{HelWad05} has the same form as our matrices $A$ and $B$.
According to~\cite[p.~398]{HelWad05}, such behavior of the condition numbers in
the CSM is expected, since the matrices ``resemble a discretization of the
kernel of a single-layer potential whose inverse is the Laplacian operator''.
Analyses of the invertibility of matrices in other applications of the CSM can
be found in~\cite{OgaOkaSugAma03}.

\begin{figure}
{\centering
\includegraphics[width=0.5\linewidth]{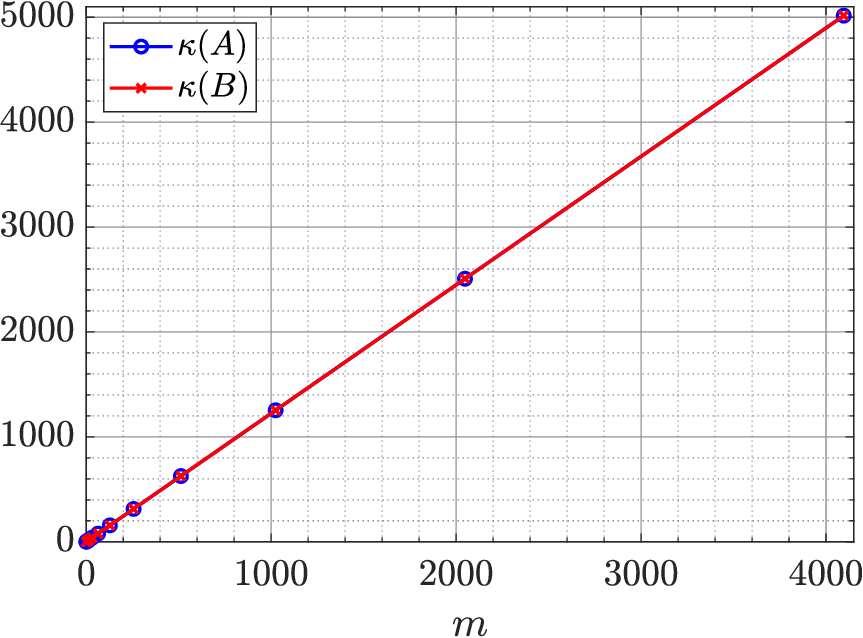}

}
\caption{2-norm condition numbers of $A \in \R^{m,m}$ and $B \in \R^{m/2, m/2}$
from~\eqref{eq:A} and~\eqref{eq:B} for $q=1/3$ as functions of $m = 2^k$, $k =
2,3, \ldots, 12$.} \label{fig:cond}
\end{figure}

From a numerical point of view, solving~\eqref{eqn:B1-system} is clearly
preferable over solving~\eqref{eqn:Ax-e}, since $B$ has only half the size of
$A$, while both matrices are dense and have essentially the same condition
number (cf. Figure~\ref{fig:cond}).  We will solve the linear algebraic
system~\eqref{eqn:B1-system} using iterative methods that require one
matrix-vector product with~$B$ in every step; see Section~\ref{sec:GMRES}.
Because of the structure of the entries of~$B$, this multiplication can
be performed using the Fast Multipole Method (FMM)~\cite{GR}.  Using the
definition of $B$ in~\eqref{eq:B}, each matrix-vector product of the form
$B{\bf y} = A_{11}{\bf y}+A_{12}J_{m/2}{\bf y}$ requires two applications of
the FMM.  The following result shows that $B$ can be written in a form so that
only one application of the FMM is required.

\begin{lemma} \label{lem:B_entries}
The entries of $B$ are given by
\[
b_{ij} =
\begin{cases}
-\log(2r\sqrt{z_{i}}), &  i=j,\\
-\log|z_{i}-z_{j}|, & i\ne j,\quad 1\le i,j\le m/2,
\end{cases}
\]
where $z_{i} \coloneq (w_{i}-1/2)^2$ for $i=1,\dots, m/2$.
\end{lemma}

\begin{proof}
First note that by definition the entries of $A_{11}$ are given by
\[
a_{ij} =
\begin{cases}
-\log r, &  i=j,\\
-\log|w_{i}-w_{j}|, & i\ne j,\quad 1\le i,j\le m/2,
\end{cases}
\]
and the entries of $A_{12}$ are given by
\[
\hat a_{ij} = -\log|w_{i}-w_{m/2+j}|, \quad 1\le i,j\le m/2.
\]
By construction, the $w_{j}$, $j=1,2,\ldots,m$, are real numbers in the
interval $(0,1)$, which are symmetric about $1/2$.  Further, we have
\begin{equation*}
w_{m/2+j} = 1-q+w_{j}, \quad w_{m/2+1-j}+w_{j} = q,\quad j=1,\dots, m/2.
\end{equation*}
Thus,
\[
\hat a_{ij} = -\log|w_{i}-w_{j}-1+q|, \quad 1\le i,j\le m/2,
\]
and hence the entries $\tilde a_{ij}$ of the matrix $A_{12}J_{m/2}$ are given
by
\[
\tilde a_{ij} = \hat a_{i,m/2+1-j}
= -\log|w_{i}-w_{m/2+1-j}-1+q|
= -\log|w_{i}+w_{j}-1|, \quad 1\le i,j\le m/2.
\]
Finally, by~\eqref{eq:B}, the entries $b_{ij}$ of $B$ are given for $i=j$ by
\[
b_{ii}=a_{ii} + \tilde a_{ii} = -\log r-\log|2w_{i}-1|
= -\log\left|2r(w_{i}-1/2)\right|,
\]
and for $i\ne j$ by
\[
b_{ij}=a_{ij}+\tilde a_{ij}=-\log|w_{i}-w_{j}|-\log|w_{i}+w_{j}-1|,
=-\log|(w_{i}-1/2)^2-(w_{j}-1/2)^2|,
\]
which completes the proof.
\end{proof}

From this lemma we have
\begin{equation}\label{eq:Bmat}
B =-
{\small \begin{bmatrix}
\log(2r\sqrt{z_{1}}) & \log|z_{1}-z_{2}|     &\cdots &\log|z_{1}-z_{m/2-1}| 
&\log|z_{1}-z_{m/2}| \\
\log|z_{2}-z_{1}| &\log(2r\sqrt{z_{2}}) &\cdots &\log|z_{2}-z_{m/2-1}| & 
\log|z_{2}-z_{m/2}| \\
\vdots & \vdots & \ddots & \vdots & \vdots \\
\log|z_{m/2-1}-z_{1}| &\log|z_{m/2-1}-z_{2}| & \cdots &\log(2r\sqrt{z_{m/2-1}}) 
 &\log|z_{m/2-1}-z_{m/2}| \\
\log|z_{m/2}-z_{1}| & \log|z_{m/2}-z_{2}|  &\cdots &\log|z_{m/2}-z_{m/2-1}| 
&\log(2r\sqrt{z_{m/2}}) \\
\end{bmatrix},}
\end{equation}
and we can easily see that $B$ is symmetric (as already mentioned above), but
not centrosymmetric.

\begin{theorem} \label{thm:B}
For $k \geq 2$, the entries $b_{ij}$ of $B$ satisfy
\begin{equation} \label{eqn:bound_bij}
\log \frac{1}{q (1-q)} < \log \frac{1}{(1-q) (q - 2r)} \leq b_{ij}
\leq \log \frac{1}{r (1 - 2q + 2r)} < \log \frac{1}{r (1-2q)},
\end{equation}
and they decay away from the diagonal in each row and column, i.e.,
\begin{align}
b_{i,1} < \ldots < b_{i, i-1} < b_{i,i}, \quad
b_{i,i} > b_{i, i+1} > \ldots > b_{i,m/2}, \label{eqn:decay_row_B} \\
b_{1,j} < \ldots < b_{j-1, j} < b_{j,j}, \quad
b_{j,j} > b_{j+1, j} > \ldots > b_{m/2,j}, \label{eqn:decay_col_B}
\end{align}
for $1 \leq i, j \leq m/2$.
\end{theorem}

\begin{proof}
In~\eqref{eqn:bound_bij}, we need $q - 2r > 0$, which is equivalent to $r <
q/2$ and thus to $k \geq 2$.

\emph{Decay.} Since
\begin{equation} \label{eqn:wk_first_half}
0 < r=w_1 < w_2 < \ldots < w_{m/2}=q-r < q < \frac{1}{2},
\end{equation}
we have
\begin{equation}\label{eq:zk-dec}
1/4 > z_1 > z_2 > \ldots > z_{m/2} > 0.
\end{equation}
This implies
\begin{equation*}
\abs{z_i - z_1} > \abs{z_i - z_2} > \ldots > \abs{z_i - z_{i-1}},
\quad
\abs{z_i - z_{i+1}} < \ldots < \abs{z_i - z_{m/2}},
\end{equation*}
and equivalently
\begin{equation*}
b_{i,1} < \ldots < b_{i, i-1}, \quad b_{i, i+1} > \ldots > b_{i,m/2}.
\end{equation*}
It remains to show $b_{i, i \pm 1} < b_{ii}$, which is equivalent to $\abs{z_i 
- z_{i \pm 1}} > 2 r
\sqrt{z_i}$.
Note that
\begin{equation*}
\abs{z_i - z_j} = \abs{w_i - w_j} \abs{w_i + w_j - 1} = \abs{w_i - w_j} (1 -
(w_i + w_j)).
\end{equation*}
Since $\abs{w_i - w_{i-1}} > r$ and $w_i + w_{i-1} < 2 w_i < 1$, we have
$\abs{z_i - z_{i-1}} > r (1 - 2 w_i) = 2 r \sqrt{z_i}$,
i.e., $b_{i, i-1} < b_{ii}$.
Since $\abs{w_i - w_{i+1}} > 2r$, we have
\begin{equation*}
\abs{z_i - z_{i+1}}
> 2r (1 - (w_i + w_{i+1})) > r (1 - 2 w_i).
\end{equation*}
The last estimate is equivalent to $1 > 2 w_{i+1}$, which holds
by~\eqref{eqn:wk_first_half}.  This, together with
$2 r \sqrt{z_i} = r \abs{2 w_i - 1} = r (1 - 2 w_i)$,
establishes $b_{ii} > b_{i, i+1}$.
We thus have shown~\eqref{eqn:decay_row_B}.
Since $B$ is symmetric, \eqref{eqn:decay_row_B} is equivalent
to~\eqref{eqn:decay_col_B}.

\medskip

\emph{Bounds for $b_{ij}$.}
Let $1 \leq i, j \leq m/2$.
To establish the upper bound, it is enough to show $b_{ii} \leq \log
\frac{1}{r(1 - 2q + 2r)}$, which is equivalent to $1 - 2q + 2r \leq 2
\sqrt{z_i} = 1 - 2
w_i$ and to $w_i \leq q-r$, which holds by~\eqref{eqn:wk_first_half}.
Next, we show the lower bound for $b_{ij}$.  It follows from~\eqref{eq:zk-dec}
that
\[
|z_i-z_j|\le|z_1-z_{m/2}|.
\]
By~\eqref{eqn:wk_first_half}, we have $z_1=(0.5-w_1)^2=(0.5-r)^2$ and $z_{m/2}=(0.5-w_{m/2})^2=(0.5-q+r)^2$. Thus
\[
|z_i-z_j|\le\left|(0.5-r)^2-(0.5-q+r)^2\right|=|(q-2r)(1-q)|
\]
and hence, for $i\ne j$,
\[
b_{ij}=-\log|z_i-z_j|\ge-\log|(q-2r)(1-q)|.
\]
The latter inequality holds also for $b_{ii}$ since it is the largest element
in its row.
The rest is clear.
\end{proof}

Let $\hat{\mathbf{y}}=B\mathbf{y}$ where $\mathbf{y}=[y_{1},\dots,y_{m/2}]^T$
and $\hat{\mathbf{y}}=[\hat y_{1},\dots,\hat y_{m/2}]^T$.  In general,
computing the vector $\hat{\mathbf{y}}$ requires $O(m^2)$ operations.  However,
for real $\mathbf{y}$,
the form of $B$ in~\eqref{eq:Bmat} yields
\begin{equation}\label{eq:hyj}
\hat y_j
= -\log(2r\sqrt{z_j}) y_j - \Re \bigg( \sum_{{\begin{subarray}{l}\ell=1 \\
\ell\ne
j\end{subarray}}}^{m/2}y_\ell\log(z_{j}-z_{\ell}) \bigg), \quad
j=1,2,\ldots,m/2,
\end{equation}
and hence computing the vector $\hat{\mathbf{y}}$ requires one application of
the FMM, which uses only $O(m)$ operations; see~\cite{GR}.  In MATLAB, the sum
in~\eqref{eq:hyj} can be computed fast and efficiently using the MATLAB function
\verb|cfmm2dpart| from the toolbox \verb|FMMLIB2D|~\cite{Gre-Gim12}.  Using the
input parameters ${\bf y}$, ${\bf z}=[z_{1},\dots,z_{m/2}]^T$ and $r$, the
vector $\hat{\mathbf{y}}=B\mathbf{y}$ can be computed by calling the following
MATLAB function:

\begin{verbatim}
function yt = By_eval(y, z, r, iprec)
a(1, :) = real(z); a(2, :) = imag(z);
[U] =  cfmm2dpart(iprec,length(z),a,1,y(:).',0,[],1);
yt  = -(log(2*r*sqrt(abs(z.'))).*y + real(U.pot).');
end
\end{verbatim}

\noindent Here, \verb|iprec| is the precision flag for the FMM.  In our
computations, we use \verb|iprec=4|, which means that the tolerance for the FMM
is $0.5\times 10^{-12}$.

\subsection{Iteratively solving the linear algebraic systems}\label{sec:GMRES}

As mentioned above, the matrix $B$ is symmetric and nonsingular.  Thus, we can
apply the MINRES method~\cite{PaiSau75} (which is well defined for symmetric
nonsingular matrices) and the GMRES method~\cite{SaaSch86} (which is well
defined for all nonsingular matrices) in order to solve the system
\eqref{eqn:B1-system} iteratively\footnote{Numerical experiments suggest that
the matrix $A$ and hence also $B$ are positive definite. We did not prove
this, but we performed numerical experiments also with the CG
method~\cite{HesSti52} applied to the system~\eqref{eqn:B1-system}, and in these
experiments the CG residual norms behaved similarly to those of MINRES.}. The
two methods are mathematically equivalent and minimize the Euclidean residual
norm over a sequence of Krylov subspaces in every step. More precisely, when
started with the initial vector $\mathbf{y}_0=0$, they generate iterates
\begin{equation*}
\mathbf{y}_j\in {\cal K}_j(B,\mathbf{e})={\rm span}\{\mathbf{e},B
\mathbf{e},\dots, B^{j-1}\mathbf{e}\},
\end{equation*}
such that
\begin{equation*}
\|\mathbf{e}-B\mathbf{y}_j\|_2=\min_{\mathbf{z}\in {\cal
K}_j(B,\mathbf{e})}
\|\mathbf{e}-B\mathbf{z}\|_2.
\end{equation*}
Each method is based on generating an orthonormal basis of the Krylov subspaces
${\cal K}_j(B,\mathbf{e})$ for $j=1,2,\dots$, and this process requires one
matrix-vector product with $B$ in every step, which can be computed using the
FMM as described above.

The essential difference between the two methods is that MINRES uses the
symmetry of the system matrix in order to generate the orthonormal Krylov
subspace bases with short (3-term) recurrences, while GMRES is based on full
recurrences.  Thus, the computational cost of a MINRES step (in terms of memory
requirements and floating point operations) is significantly lower than of a
GMRES step.  However, methods based on short recurrences explicitly perform the
orthogonalization only with respect to a few recent vectors.  In finite
precision computations this can lead to a much faster overall loss of
orthogonality, and even a loss of rank in the computed Krylov subspace
``basis''.  Such loss of rank corresponds to a delay of convergence; see,
e.g.,~\cite{MeuStr06} or~\cite[Section~5.9]{LieStr13} for comprehensive
analyses of this phenomenon and further references.

In the application in this paper the observed delay of convergence is so
severe that MINRES is not competitive with GMRES, although it is based on
short recurrences.

\begin{example}\label{ex:GMvsMR}
We apply the MATLAB implementations of MINRES and GMRES with the initial vector
$\mathbf{y}_0=0$ to linear algebraic systems of the form \eqref{eqn:B1-system}
for $q=1/3$ with $m=2^{14}$ and $m=2^{16}$.  The matrix-vector products with
$B$ are performed using the function \verb|By_eval| described above, and hence
our function calls to the iterative methods are

\begin{verbatim}
minres(@(y) By_eval(y, z, r, 4),ones(m/2,1),tol,MAXIT);
gmres(@(y) By_eval(y, z, r, 4),ones(m/2,1),[],tol,MAXIT);
\end{verbatim}

\noindent Our tolerance for the relative residual norm is \verb|tol=1e-12| and
we use \verb|MAXIT=400| as the maximal number of iterations, but this number is
not reached in our experiment.  In Figure~\ref{fig:MR-GM} we plot the relative
residual norms of the two methods. In exact arithmetic these norms are
identical.  In our finite precision computation we observe that MINRES suffers
from a significant delay of convergence.  This delay not only leads to many
more iterative steps until the tolerance is reached, but also to a much longer 
computation time in comparison with GMRES:

\begin{center}
\begin{tabular}{c|c|c}
& GMRES time (\textup{s}) & MINRES time (\textup{s})\\ \hline
$m=2^{14}$ & 1.91 & 3.14 \\
$m=2^{16}$ & 8.00 & 15.46
\end{tabular}
\end{center}
\end{example}

\begin{figure}
{\centering
\includegraphics[width=0.475\linewidth]{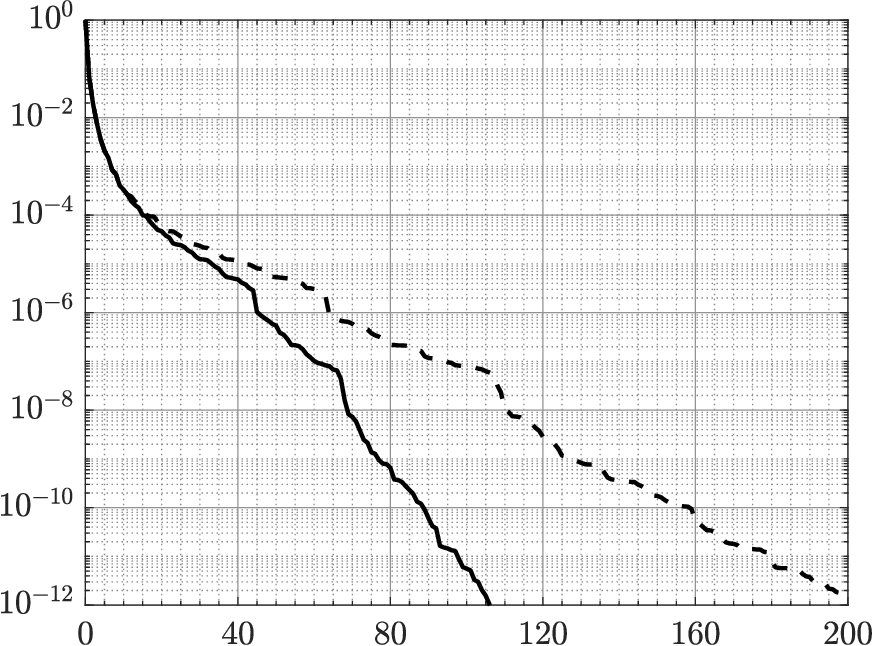}
\hfill
\includegraphics[width=0.475\linewidth]{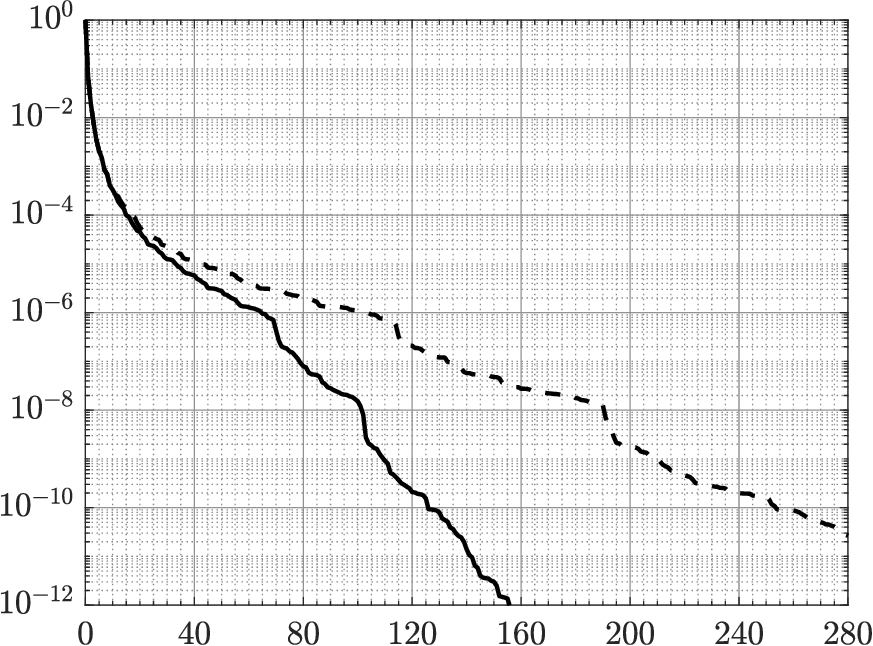}

}
\caption{Relative residual norms of GMRES (solid) and MINRES (dashed) for the
linear algebraic systems \eqref{eqn:B1-system} for $q=1/3$ with $m=2^{14}$
(left) and $m=2^{16}$ (right) in Example~\ref{ex:GMvsMR}.} \label{fig:MR-GM}
\end{figure}

\medskip
As a consequence of these numerical observations we have decided to use GMRES
in all our experiments.

We will now describe a problem-adapted technique to precondition the system
\eqref{eqn:B1-system}, which will lead to an even faster convergence of GMRES.
Consider the matrix $A$ in \eqref{eq:A} for a fixed size $m=2^k$.  Because of
the symmetric distribution of the $w_j$, $j=1,\dots,m$, for any fixed
$j=1,\dots,k-1$ we can write this matrix in the block form
\begin{equation}\label{eq:A-D}
A =-
\begin{bmatrix}
D           &D_{12}     &\cdots &D_{1p}  \\
D_{21}      &D          &\ddots & \vdots \\
\vdots      &\ddots     &\ddots & D_{p-1,p} \\
D_{p1}      & \cdots & D_{p,p-1} & D \\
\end{bmatrix}, \quad\mbox{where $D\in \R^{2^j\times 2^j}$, and $p=2^{k-j}$.}
\end{equation}
By construction, the entries of $A$ decay row- and column-wise with their
distance from the diagonal; see Theorem~\ref{thm:A}.  Hence the ``block diagonal
part'' $P_m:= \diag(D,\dots,D)$ contains the largest entries of $A$.  This
block diagonal matrix is easy to invert when $j$ is not too large, since it
requires just one inversion of $D$.  Hence it can be used as a preconditioner
for the linear algebraic system~\eqref{eqn:Ax-e}.

From~\eqref{eq:B} we have $B=A_{11}+A_{12}J_{m/2}$, where $A_{11}$ is the
leading $(\frac{m}{2}\times\frac{m}{2})$-block of $A$, and $A_{12}$ is the
upper off-diagonal block, which overall contains smaller entries than $A_{11}$.
The ``block diagonal part'' of $A_{11}$ contains $p/2$ copies of $D$.  We use
this matrix
\begin{equation*}
P_{m/2} := \diag(D,\dots,D) \in\R^{\frac{m}{2}\times\frac{m}{2}}
\end{equation*}
as our preconditioner for the system \eqref{eqn:B1-system} with~$B$, i.e.,
instead of \eqref{eqn:B1-system} we apply GMRES to the system
\begin{equation}\label{eq:PBsys}
P_{m/2}^{-1}B{\mathbf y}=P_{m/2}^{-1}{\mathbf e}.
\end{equation}
The effectiveness of this approach is illustrated next.

\begin{example} \label{ex:GM_preconditioner}
We consider the same linear algebraic systems as in Example~\ref{ex:GMvsMR}.
Figure~\ref{fig:GMRES-precon} shows the relative residual norms of GMRES
applied to the system $B\mathbf{y}=\mathbf{e}$ (solid) and the preconditioned
system $P_{m/2}^{-1}B\mathbf{y}=P_{m/2}^{-1}\mathbf{e}$ (dotted), where we have
used $j=4$ for the preconditioner $P_{m/2}$, and have inverted the matrix
$D\in\R^{16\times 16}$ explicitly using MATLAB's \verb|inv| function.
The solid curves in Figure~\ref{fig:GMRES-precon} are the same as in
Figure~\ref{fig:MR-GM}.  We observe that the preconditioning reduces the number
of GMRES iterations to reach the relative residual norm tolerance of $10^{-12}$
by more than $50\%$, and that the required times are reduced accordingly:

\begin{center}
\begin{tabular}{c|c|c}
& GMRES time (\textup{s}) & Preconditioned GMRES time (\textup{s})\\ \hline
  $m=2^{14}$ & 1.91 & 0.98 \\
  $m=2^{16}$ & 8.00 & 3.83
\end{tabular}
\end{center}
\medskip

\noindent The effect of the preconditioner on the matrix condition number can
be seen in Figure~\ref{fig:AB-cond}, where we show the condition numbers of
$P_m^{-1}A$ and $P_{m/2}^{-1}B$ as functions of $m$ for several values of $j$.
The results presented in Figure~\ref{fig:AB-cond-2} suggest that
\[
\frac{\kappa(P_m^{-1}A)}{\kappa(A)} \approx \frac{1}{2^{j+1/2}}
\quad\mbox{and}\quad
\frac{\kappa(P_{m/2}^{-1}B)}{\kappa(B)} \approx \frac{1}{2^{j+1/2}}.
\]
Clearly, the condition numbers of the preconditioned matrices decrease with
increasing $j$, while a larger~$j$ leads to higher computational costs for
applying the preconditioner.  In our experiments we found that $j=4$
represents a good trade-off value.
For this value we have $\frac{\kappa(P_m^{-1}A)}{\kappa(A)} \approx 0.045$ and
$\frac{\kappa(P_{m/2}^{-1}B)}{\kappa(B)} \approx 0.045$.
\end{example}

\begin{figure}
{\centering
\includegraphics[width=0.475\linewidth]{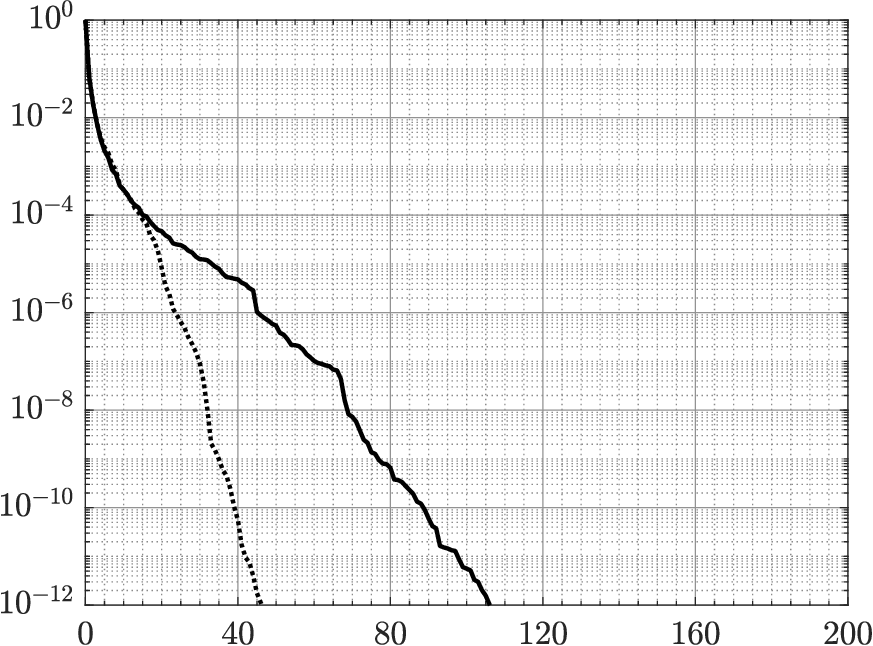}\hfill
\includegraphics[width=0.475\linewidth]{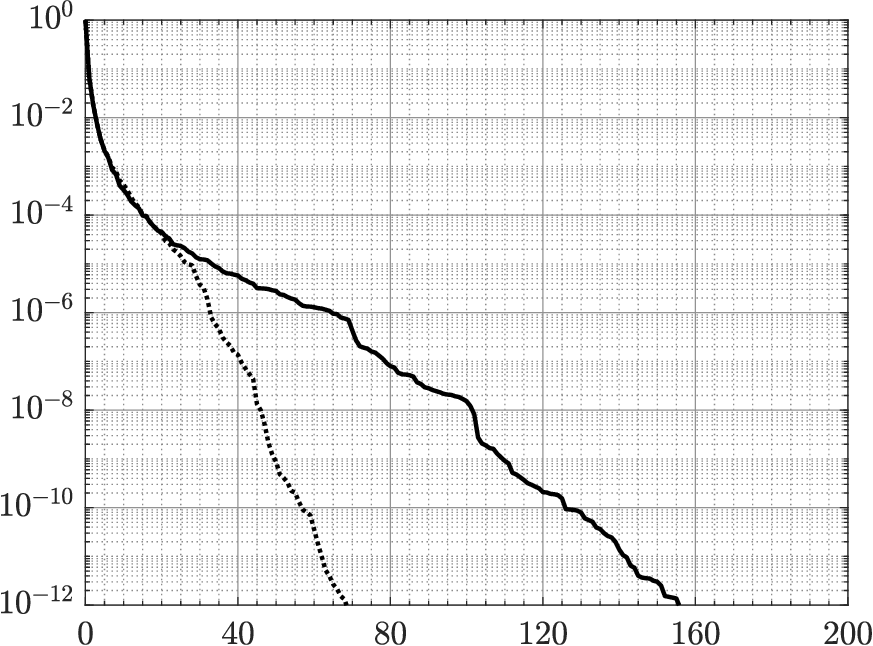}
}
\caption{Relative residual norms of GMRES (solid) and preconditioned GMRES
(dotted) with $m=2^{14}$ (left) and $m=2^{16}$ (right) in
Example~\ref{ex:GM_preconditioner}.} \label{fig:GMRES-precon}
\end{figure}

\begin{figure}
{\centering
\includegraphics[width=0.475\linewidth]{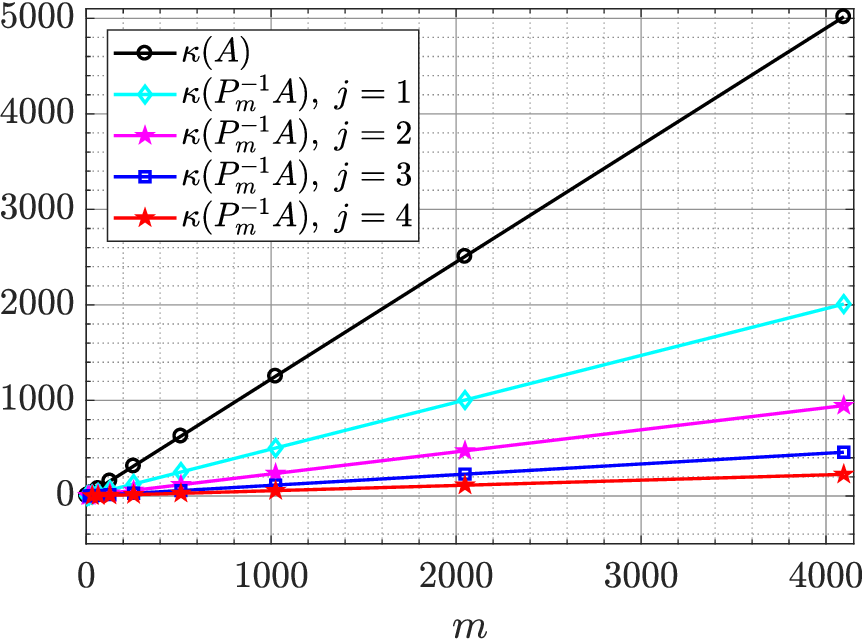}
\hfill
\includegraphics[width=0.475\linewidth]{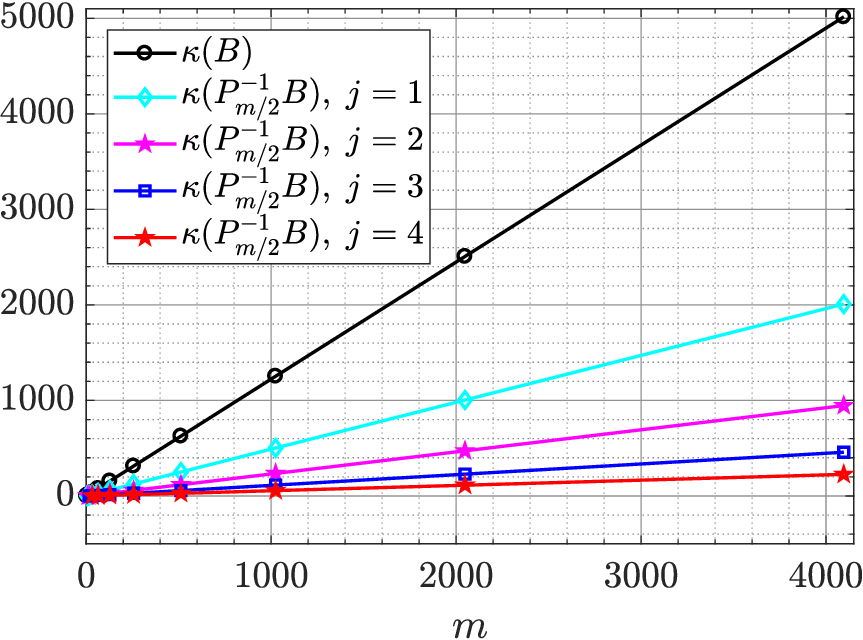}
}
\caption{2-norm condition numbers of the matrices $P_m^{-1}A \in \R^{m,m}$ and
$P_{m/2}^{-1}B \in \R^{m/2, m/2}$ as functions of $m = 2^k$ for $q=1/3$,
$j=1,2,3,4$, and $k=j+1,j+2, \ldots,12$.}
\label{fig:AB-cond}
\end{figure}

\begin{figure}
{\centering
\includegraphics[width=0.475\linewidth]{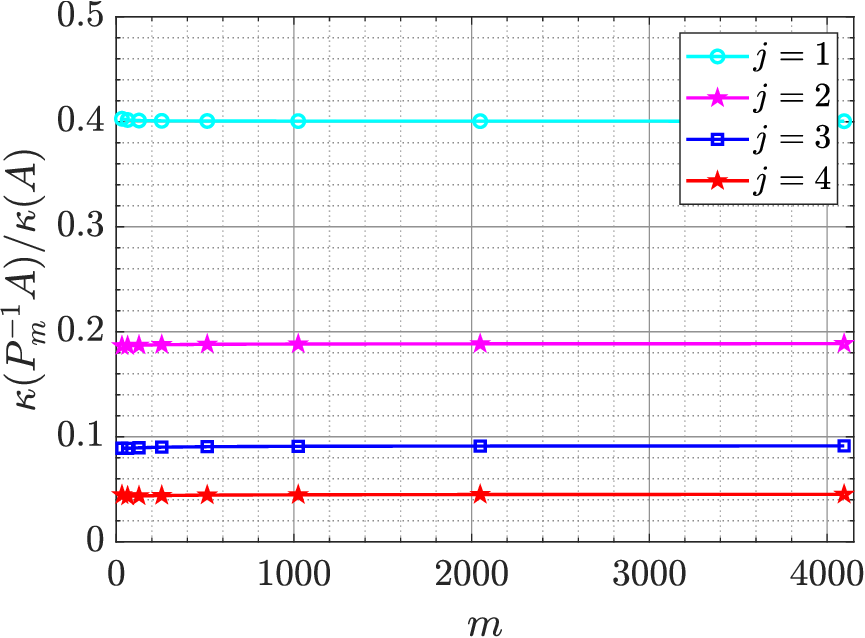}
\hfill
\includegraphics[width=0.475\linewidth]{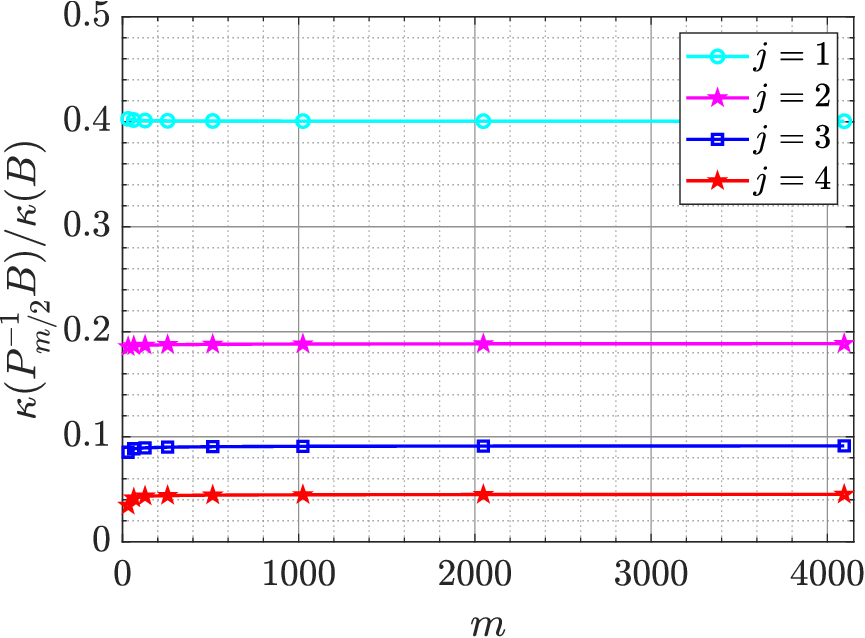}
}
\caption{$\kappa(P_m^{-1}A)/\kappa(A)$ and $\kappa(P_{m/2}^{-1}B)/\kappa(B)$ as
functions of $m = 2^k$ for $q=1/3$, $j=1,2,3,4$, and $k=5,6, \ldots,12$.}
\label{fig:AB-cond-2}
\end{figure}

Using the preconditioner $P_{m/2}$, the above MATLAB function \verb|By_eval| is
modified as follows:

\begin{verbatim}
function yt = By_peval(y, z, r, iprec, invD)
a(1, :) = real(z); a(2, :) = imag(z);
[U] =  cfmm2dpart(iprec,length(z),a,1,y(:).',0,[],1);
yt  = -(log(2*r*sqrt(abs(z.'))).*y + real(U.pot).');
[dD, dD] = size(invD);
for k=1:length(z)/dD
     yt(dD*(k-1)+1:dD*k) = invD*yt(dD*(k-1)+1:dD*k);
end
end
\end{verbatim}

\subsection{Computing the logarithmic capacity of generalized Cantor sets}
\label{sec:NumExGenC}

We now present the results of numerical computations with the method described
above for computing the logarithmic capacity of generalized Cantor sets.  Our
method requires the parameters $q$ (for the definition of the generalized
Cantor set) and $k$ (for the level of the approximation) as its only inputs;
see~\eqref{eq:Ek}.  These parameters completely determine the preconditioned
linear algebraic system~\eqref{eq:PBsys} that we solve with GMRES, where the
matrix-vector products are computed using the FMM as described above.  As in
the examples above, the tolerance for the relative GMRES residual norm is
$10^{-12}$.  GMRES then returns a computed approximation $\widetilde{\mathbf y}
\approx {\mathbf y}$, and we report the value $\exp(-1/(2{\mathbf e}^T
\widetilde{\mathbf y}))$ as $\capp(D_k)$ in our tables below.

\begin{example}[Classical Cantor middle third set]\label{ex:third}
We consider the classical Cantor middle third set, i.e., the set~\eqref{eq:E}
with $q=1/3$ in~\eqref{eq:Ek}.
Table~\ref{tab:ccs-n} gives the approximate values of $\capp(D_k)$ for
$k=5,6\dots,20$ computed by the method presented in this paper, as well as the
computation time (in seconds), and the number of GMRES iteration steps.

\begin{table}[t]
\caption{Computed approximations of $\capp(D_k)$ for $q = 1/3$,
timings, and number of GMRES iteration steps using the new proposed method for 
Example~\ref{ex:third}.} \label{tab:ccs-n}
\begin{center}
\begin{tabular}{r|r|c|r|r}
\hline
$k$ & $m=2^k$  & $\capp(D_k)$ & time (s) & iter\\
\hline
 5  & 32      & 0.227457816902705 & 0.03 & 4 \\
 6  & 64      & 0.224254487425132 & 0.03 & 7 \\
 7  & 128     & 0.222633059381908 & 0.06 & 10 \\
 8  & 256     & 0.221808427761487 & 0.11 & 15 \\
 9  & 512     & 0.221387991441743 & 0.16 & 19 \\
 10 & 1024    & 0.221173357505459 & 0.22 & 22\\
 11 & 2048    & 0.221063713734092 & 0.27 & 25 \\
 12 & 4096    & 0.221007684178946 & 0.40 & 32 \\
 13 & 8192    & 0.220979047273590 & 0.68 & 39 \\
 14 & 16384   & 0.220964409542387 & 1.00 & 47 \\
 15 & 32768   & 0.220956927135913 & 1.70 & 57 \\
 16 & 65536   & 0.220953102245645 & 3.37 & 69 \\
 17 & 131072  & 0.220951146997627 & 7.95 & 81 \\
 18 & 262144  & 0.220950147487058 & 16.79 & 97 \\
 19 & 524288  & 0.220949636541913 & 37.79 & 119 \\
 20 & 1048576 & 0.220949375348718 & 96.98 & 143 \\
\hline
\end{tabular}
\end{center}
\end{table}

\begin{table}[t]
\caption{Computed approximations of $\capp(E_k)$ and $\capp(D_k)$ for $q = 
1/3$, and timings using the BIE Method~\cite{LSN} for Example~\ref{ex:third}.} 
\label{tab:ccs-o}
\begin{center}
\begin{tabular}{r|r|c|r|c|r}
\hline
$k$ & $m=2^k$  & $\capp(E_k)$ & time (s) & $\capp(D_k)$ & time (s) \\
\hline
 5  & 32      & 0.221938129124324 & 9.95     & 0.227918836283900 & 4.69 \\
 6  & 64      & 0.221454205006181 & 19.24    & 0.224486551122397 & 12.67 \\
 7  & 128     & 0.221207178734289 & 47.21    & 0.222750783871879 & 40.66 \\
 8  & 256     & 0.221080995391656 & 131.18   & 0.221868377621828 & 111.86 \\
 9  & 512     & 0.221016516406109 & 402.66   & 0.221418578408387 & 357.10 \\
 10 & 1024    & 0.220983561713855 & 1405.82  & 0.221188978166610 & 1277.74\\
 11 & 2048    & 0.220966717159289 & 5071.78  & 0.221071694998756 & 4351.78 \\
 12 & 4096    & 0.220958106742622 & 18207.76 & 0.221011763144866 & 15509.24 \\
\hline
\end{tabular}
\end{center}
\end{table}

We compare these results with the computed values and the corresponding timings
for the BIE method from~\cite{LSN}. Approximating $\capp(E_k)$ with the BIE
method requires a preliminary conformal map to ``open up'' the intervals of 
$E_k$ to obtain a compact set of the same capacity, but bounded by smooth 
Jordan curves.  Then the method is used to compute the capacity of this new 
set. For approximating $\capp(D_k)$ with the BIE method, no preliminary 
conformal map is needed since $\partial D_k$ is smooth.  For both cases, we 
take $n=2^6$ discretization points on each of the $m=2^k$ boundary curves, and 
the obtained results are presented in Table~\ref{tab:ccs-o}.
(The value $n = 2^6$ was chosen, since with this value the BIE method yields
the logarithmic capacity of a single disk and of two disks with equal radius
with a relative error of order $10^{-16}$.)
Further, by Tables~\ref{tab:ccs-n} and~\ref{tab:ccs-o}, there is a good
agreement between the approximations of $\capp(D_k)$ obtained by the BIE method
from~\cite{LSN} and the method presented in this paper. This agreement
improves as $k$ increases.
As described in Remark~\ref{rmk:savings}, the new method is
significantly more efficient, since it uses only a single charge point for each
component of $D_k$. In addition, we have used the special (centrosymmetric)
structure of the system matrices as well as a preconditioner for GMRES to
speed up the computations.

The estimate of $\capp(E)$ in~\cite{LSN} was obtained by extrapolation from the
computed values for $\capp(E_k)$.  Using the same approach here, we start by
noting that the differences
\begin{equation*}
d_k= \capp(D_k)-\capp(D_{k+1}),\quad k=5,6,\dots,19,
\end{equation*}
decrease linearly on a logarithmic scale. We store these $15$ values in the
vector $d$, and use the MATLAB command {\tt p=polyfit(5:19,log(d),1))} to
compute a linear polynomial $p(x)=p_1x+p_2$ of best approximation in the least 
squares sense for the values $\log(d_k)$.  The computed coefficients are
\begin{equation*}
p_1=-0.671894676421546,\quad p_2=-2.39546038319728.
\end{equation*}
Starting with our computed approximation of $\capp(D_{20})$ we can then
approximate $\capp(D_k)$ for $k\geq 21$ by extrapolation, i.e.,
\begin{equation*}
\capp(D_k)\approx \capp(D_{20})-\sum_{j=20}^{k-1}\exp(p(j)),\quad k\geq 21.
\end{equation*}
Since $\exp(p(52))<10^{-16}$, we use this formula with $k=52$ for our final 
estimate of $\capp(E)$ which is shown in Table~\ref{tab:middle-third}.
Applying the same extrapolation approach to the values obtained 
by the BIE method (see Table~\ref{tab:ccs-o}) yields the estimates of 
$\capp(E)$ presented in Table~\ref{tab:middle-third}.

\begin{table}[t]
\caption{Estimates of the logarithmic capacity of the classical
Cantor middle third set.  Matching digits are underlined.} 
\label{tab:middle-third}
\begin{center}
\begin{tabular}{l|l}
\hline
New Method & \underline{0.220949}103628452 \\ \hline
BIE Method (based on $\capp(E_k)$) & \underline{0.220949}114469744 \\ \hline
BIE Method (based on $\capp(D_k)$) & \underline{0.220949}728829335 \\ \hline
Ransford~\cite{Ransford2010} & \underline{0.220949}102189525665\\ \hline
Kr\"{u}ger and Simon~\cite{KruSim15} & \underline{0.220949}98647421 \\
\hline
\end{tabular}
\end{center}
\end{table}

Our estimate (obtained with the new proposed method) agrees in its first eight 
significant digits with the one that was ``strongly suggested''
by Ransford~\cite[p.~568]{Ransford2010}. All these estimates agree in their
first six significant digits, and they are all contained in the interval
\begin{equation*}
[0.22094810685,0.22095089228],
\end{equation*}
which according to Ransford and Rostand~\cite{RanRos07} contains $\capp(E)$.
\end{example}

\begin{example}[Generalized Cantor set]\label{ex:generalized}
We consider now the numerical approximation of $\capp(E)$ for general 
$q\in(0,0.5)$.
The limiting cases are $E=\{0,1\}$ with $\capp(E)=0$ (for $q=0$)
and $E=[0,1]$ with $\capp(E)=0.25$ (for $q=0.5$).
For several values of $q$, we approximate the values of $\capp(D_k)$ with the
CSM
and extrapolate these values to obtain an approximation of $\capp(E)$ using the
same approach as described in Example~\ref{ex:third}.  The results are given in
Table~\ref{tab:q}, where the values of $q$ are chosen from the values 
considered in~\cite[Example~4.14]{LSN} (note that $q$ here is denoted by $r$ 
in~\cite{LSN}).
We also state the values of $\capp(E)$ computed in~\cite{LSN} using a 
discretization with $k=12$ (corresponding to $m=4096$), and then extrapolation.
We observe that the estimates obtained by the two methods agree at least in
their first four significant digits (matching digits are underlined).

\begin{table}[t]
\caption{The computed approximations of $\capp(E)$ for the
generalized Cantor set; see Example~\ref{ex:generalized}.}\label{tab:q}
\begin{center}
\begin{tabular}{l|l|l}
\hline
$q$      &  BIE Method~\cite{LSN} & New Method  \\
\hline
 $4/24$  & \underline{0.1384}4418298159   & \underline{0.1384}37531550946 \\
 $6/24$  & \underline{0.1865}11016338442  & \underline{0.1865}08655120292  \\
 $8/24$  & \underline{0.2209491}94629475  & \underline{0.2209491}03628452  \\
 $9/24$  & \underline{0.233218}551525021  & \underline{0.233218}660959678  \\
 $10/24$ & \underline{0.242233}234580321  & \underline{0.242233}644605597  \\
 $11/24$ & \underline{0.2479}29630663845  & \underline{0.2479}30030139435   \\
\hline
\end{tabular}
\end{center}
\end{table}

It was suggested in~\cite{RanRos07} that the values of $\capp(E)$ can be
approximated by
\begin{equation} \label{eqn:approximation_of_capEq}
\capp(E)\approx f(q)=q(1-q)-\frac{q^3}{2}\left(\frac{1}{2}-q\right)^{3/2}.
\end{equation}
Figure~\ref{fig:cap-f} shows the graph of the function $f(q)$ as well as
our computed approximations of $\capp(E)$.  The maximum distance between
the values of
$f(q)$ and the computed approximations of $\capp(E)$ from the third column of 
Table~\ref{tab:q} is $7.56\times10^{-5}$, which is close to the results 
reported in~\cite{LSN}.

\begin{figure}[t]
{\centering
\includegraphics[width=0.6\linewidth]{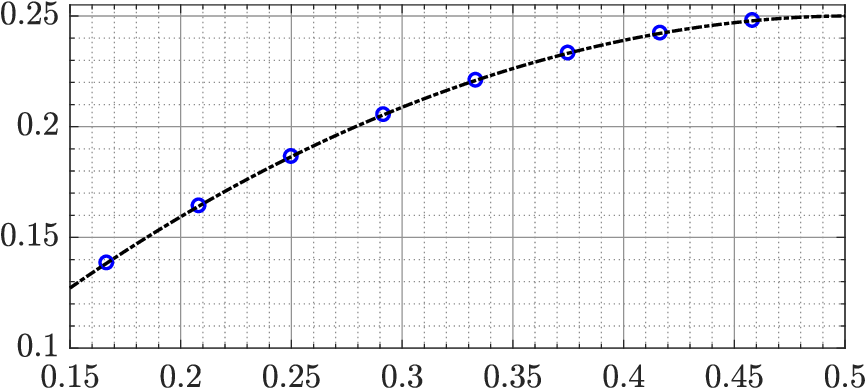}

}
\caption{The computed logarithmic capacity $\capp(E)$ (blue circles) and $f(q)$ 
from~\eqref{eqn:approximation_of_capEq} (dashed); see 
Example~\ref{ex:generalized}.}
\label{fig:cap-f}
\end{figure}
\end{example}

\section{Cantor dust}\label{sec:dust}

The Cantor dust is a generalization of the Cantor set to dimension two.
We compute the logarithmic capacity of the generalized Cantor dust, and start 
by setting up the CSM, similarly to our approach in 
Section~\ref{sec:Gen-Cantor}.

Let $q \in (0,1/2)$ and $F_0 = [0,1] \times [0,1]$.  Define recursively
\begin{equation*} 
F_k := qF_{k-1}\cup\left(qF_{k-1}+1-q\right)\cup\left(qF_{k-1}
+(1-q)\i\right)\cup\left(qF_{k-1}+(1-q)(1+\i)\right), \quad k \geq 1,
\end{equation*}
i.e., $F_k = E_k \times E_k$ with $E_k$ from~\eqref{eq:Ek}.
Then the generalized Cantor dust $F$ is defined as
\begin{equation}\label{eq:F}
F \coloneq \bigcap_{k=0}^{\infty} F_k.
\end{equation}
Note that $F_k$ consists of $m = 4^k$ closed square regions, say $S_{k,1},
S_{k,2}, \ldots, S_{k,m}$; see Figure~\ref{fig:dust} for $q=1/4$ with $k=1$
(left) and $k=2$ (right).
The diameter of each of the squares $S_{k,1},S_{k,2},\ldots,S_{k,m}$ is $2r_k$,
where
\begin{equation}\label{eq:rk-d}
r_k:=\frac{1}{\sqrt{2}}q^k.
\end{equation}
For $j=1,2,\ldots,m$, denote the center of $S_{k,j}$ by $w_{k,j}$.
To be precise, we order the points $w_{k,j}$ recursively by $w_0 = (1+\i)/2$
and
\begin{equation}
w_k = [w_{k,j}] = [q w_{k-1}, q w_{k-1} + (1-q), q w_{k-1} + (1-q) \i, q
w_{k-1} + (1-q) (1 + \i)] \in \C^{1, 4^k}
\label{eqn:wk_dust}
\end{equation}
for $k \geq 1$.
Let $D_{k,j}$ be the disk with center $w_{k,j}$ and radius $r_k$, and let
\[
D_k = \bigcup_{j=1}^m D_{k,j};
\]
see Figure~\ref{fig:dust} for $q=1/4$ with $k=1$ (left) and $k=2$ (right).

\begin{figure}
{\centering
\includegraphics[width=0.4\linewidth]{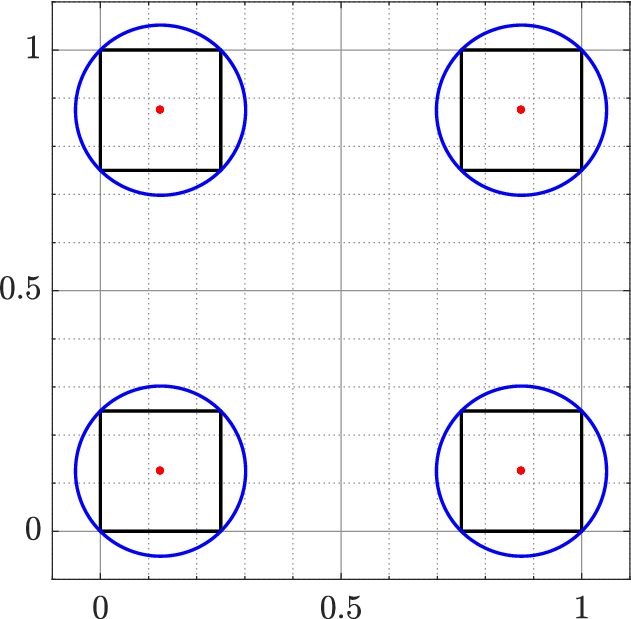} \hfill
\includegraphics[width=0.4\linewidth]{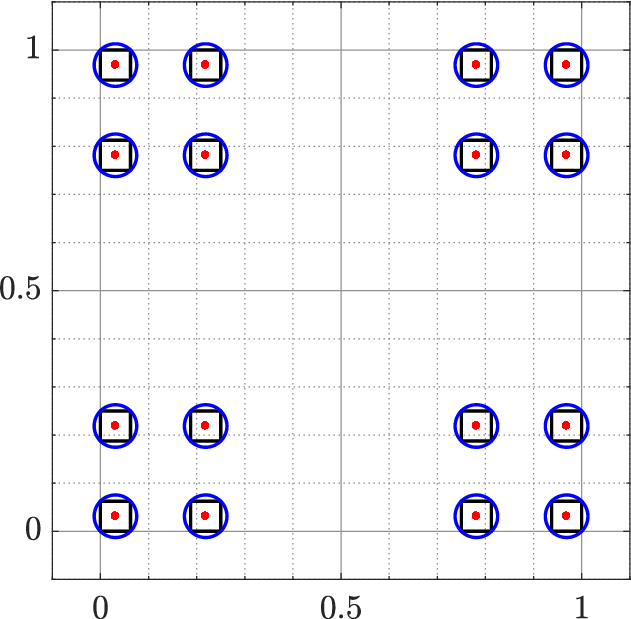}

}
\caption{The Cantor dust with $q=1/4$ for $k=1$ (left) and $k=2$ (right).}
\label{fig:dust}
\end{figure}

\begin{theorem} \label{thm:cantor_dust}
Let $F$ be the Cantor dust defined by~\eqref{eq:F}, then
\[
\bigcap_{k=0}^{\infty} D_k = F
\quad \text{and} \quad
\capp(F) = \lim_{k \to \infty} \capp(D_k).
\]
\end{theorem}

\begin{proof}
By construction, $F_k \subseteq D_k$ for all $k \in \N_0$, hence $F \subseteq
\bigcap_{k=0}^\infty D_k$.
If $z \in \C \setminus F$, then there exists $k_0 \in \N$ with
$z \notin F_{k_0}$.  In particular, $d \coloneq \dist(z, F_{k_0}) > 0$ since
$F_{k_0}$ is compact.
Since the maximal distance from a point on $\partial D_{k,j}$ to the
inscribed square $S_{k,j}$ is $r_k - \frac{1}{2} q^k = \frac{\sqrt{2} - 1}{2}
q^k \to 0$ for $k \to \infty$, we have $z \notin D_{k_1}$ for a sufficiently
large $k_1$, hence $z \notin \bigcap_{k=0}^\infty D_k$.
Finally, $\capp(F) = \lim_{k \to \infty} \capp(D_k)$
by~\cite[Theorem~5.1.3]{Ransford1995}, since $D_0 \supseteq D_1 \supseteq D_2
\supseteq \ldots$ are compact and $F = \bigcap_{k=0}^\infty D_k$.
\end{proof}

As in Section~\ref{sec:Gen-Cantor}, we first approximate 
$\capp(D_k)$ for several
$k$ with the CSM, and then extrapolate these values to obtain an approximation
of $\capp(F)$.
With the choice~\eqref{eq:rk-d} of $r_k$, and to make sure the disks $D_{k,j}$,
$j = 1, 2, \ldots, 4^k$, are disjoint, we will consider here only the case
$q<\sqrt{2}-1$.  Then
the complement of $D_k$, denoted by $G_k=(\C\cup \{\infty\}) \setminus D_k$, is an
unbounded multiply connected domain of connectivity $m = 4^k$ with
\[
\partial G_k = \partial D_{k,1}\cup\cdots\cup \partial D_{k,4^k}.
\]
We parametrize the circle $\partial D_{k,j}$ by $\eta_{k,j}(t) = w_{k,j} +
r_k e^{\i t}$, $0\le t\le 2\pi$, for $j=1,2,\ldots,4^k$.

As described in Section~\ref{sec:Gen-Cantor}, we approximate the Green's 
function of $G_k$ by a function $h_k$ of the 
form~\eqref{eqn:approx_Green_function} with~\eqref{eq:sum-p}.
The condition~\eqref{eqn:hk_zero_in_the_mean} leads to the linear algebraic 
system~\eqref{eq:sys-5a}
with the matrix $A \in \R^{m,m}$ from~\eqref{eq:A}, now with $m = 4^k$ instead
of $2^k$, which is symmetric and centrosymmetric.
Then
\begin{equation*}
\capp(D_k) \approx e^{-c_k} \quad \text{where} \quad
c_k= \frac{1}{\mathbf{e}^T A^{-1} \mathbf{e}}.
\end{equation*}
As before, we solve $A \mathbf{x} = \mathbf{e}$ to obtain $c_k$ as $c_k =
1/(\mathbf{e}^T \mathbf{x})$.
Note that the estimates of the entries of $A$ in Theorem~\ref{thm:A} do not
carry over from the case of the generalized Cantor set to the generalized
Cantor dust.

We will now fix $m = 4^k$ and drop the index $k$ in the following for
simplicity.
The matrix $A$ for the generalized Cantor dust has the same structure as the
matrix $A$ for the generalized Cantor set.  In particular, using that $A$ is
centrosymmetric, we can reduce the linear system $A \mathbf{x} = \mathbf{e}$ of
size $m$ to the system
\begin{equation*}
B \textbf{y} = \textbf{e}
\end{equation*}
of size $m/2$, where $B$ is given by~\eqref{eq:B}, and obtain
\[
c = \frac{1}{2 \mathbf{e}^T \mathbf{y}}
\approx \frac{1}{2\mathbf{e}^T \widetilde{\mathbf{y}}},
\]
where $\widetilde{\mathbf{y}}$ is a computed approximate solution to $B
\mathbf{y} = \mathbf{e}$, as described in Section~\ref{sect:system}.
The following analog of Lemma~\ref{lem:B_entries} for the generalized Cantor
dust allows to compute a multiplication with $B$ with a single application of
the FMM.

\begin{lemma}
The entries of $B = \begin{bmatrix} b_{ij} \end{bmatrix} \in \R^{m/2, m/2}$ are
given by
\[
b_{ij} =
\begin{cases}
-\log \abs{2r_k \sqrt{z_{i}}}, &  i=j, \\
-\log|z_{i}-z_{j}|, & i\ne j, \quad 1\le i,j\le m/2,
\end{cases}
\]
where $z_{i} \coloneq (w_i - (1+\i)/2)^2$ for $i=1,\dots, m/2$, and $r_k$ is
given by~\eqref{eq:rk-d}.
\end{lemma}

\begin{proof}
The entries of the matrix $A_{11}$ are given by
\[
a_{ij} =
\begin{cases}
-\log r_k, &  i=j,\\
-\log|w_{i}-w_{j}|, & i\ne j,
\end{cases}
\]
for $1\le i,j\le m/2$.
The entries of the matrix $A_{12}$ are given by
\[
\hat a_{ij} = -\log|w_{i}-w_{m/2+j}|, \quad 1\le i,j\le m/2.
\]
Unlike the generalized Cantor set, here the points $w_j$, $j=1,2,\ldots,m$,
are complex numbers in the square domain $(0,1)\times(0,1)$.
By the definition~\eqref{eqn:wk_dust} of the points $w_{j}$, we have
\begin{equation*}
w_{m/2+j} = (1-q)\i+w_{j}, \quad w_{m/2+1-j}+w_{j} = 1+q\i,
\quad \text{for } 1\le j\le m/2.
\end{equation*}
Thus,
\[
\hat a_{ij} = -\log|w_{i}-w_{j}-(1-q)\i|, \quad 1\le i,j\le m/2,
\]
and hence the entries $\tilde a_{ij}$ of the matrix $A_{12}J_{m/2}$ are given by
\[
\tilde a_{ij} = \hat a_{i,m/2+1-j}
= -\log|w_{i}-w_{m/2+1-j}-(1-q)\i|
= -\log|w_{i}+w_{j}-1-\i|, \quad 1\le i,j\le m/2.
\]
Finally, the entries $b_{ij}$, $1\le i,j\le m/2$, of the matrix $B$ are given
for $i=j$ by
\[
b_{ii} = a_{ii} + \tilde a_{ii} = -\log r_k-\log|2w_{i}-1-\i|
= -\log \abs{2r_k(w_{i}-(1+\i)/2)}
= -\log \abs{2r_k \sqrt{z_{i}}},
\]
and for $i\ne j$ by
\[
b_{ij}=a_{ij}+\tilde a_{ij}=-\log|w_i - w_j|-\log|w_i + w_j - 1-\i|
= - \log \abs{z_i - z_j},
\]
as claimed.
\end{proof}

The matrix $B$ is symmetric but not centrosymmetric.
Similarly to the approach for the generalized Cantor set in
Section~\ref{sec:GMRES}, 
the condition number of $B$ grows linearly with $m$; see 
Figure~\ref{fig:condd}.  We solve $B \mathbf{y} = \mathbf{e}$ with GMRES and
construct a left preconditioner
for the matrices $A$ and $B$ in the case of the generalized Cantor dust.
For $j = 1, 2, \ldots, k-1$, the matrix $A$ can be written in
the block form~\eqref{eq:A-D} where now $D \in \R^{4^j \times 4^j}$ and
$p=4^{k-j}$.
Define the block-diagonal matrices
\begin{equation*}
P_m = \diag(D, \ldots, D) \in \R^{m,m} \quad \text{and} \quad
P_{m/2} = \diag(D, \ldots, D) \in \R^{m/2,m/2},
\end{equation*}
where $P_m$ contains $p=4^{k-j}$ copies of the matrix $D$, and $P_{m/2}$
contains $p/2$ copies of $D$.
Then the matrix $P_m$ is used as a preconditioner for the
system~\eqref{eqn:Ax-e},
and the matrix $P_{m/2}$ is used as a preconditioner for the
system~\eqref{eqn:B1-system}.
The condition numbers of the matrices $P_m^{-1} A$ and $P_{m/2}^{-1} B$ is
shown in Figure~\ref{fig:ABd-cond} for $j=1,2$.
In our numerical computation, we consider $j=2$ and hence the size of the
matrix $D$ is $16\times 16$.  In this case, we have $\kappa(P_m^{-1} A) \approx
0.077\,\kappa(A)$ and $\kappa(P_{m/2}^{-1} B) \approx 0.077\,\kappa(B)$;
see Figure~\ref{fig:ABd-cond-2}.

\begin{figure}
{\centering
\includegraphics[width=0.45\linewidth]{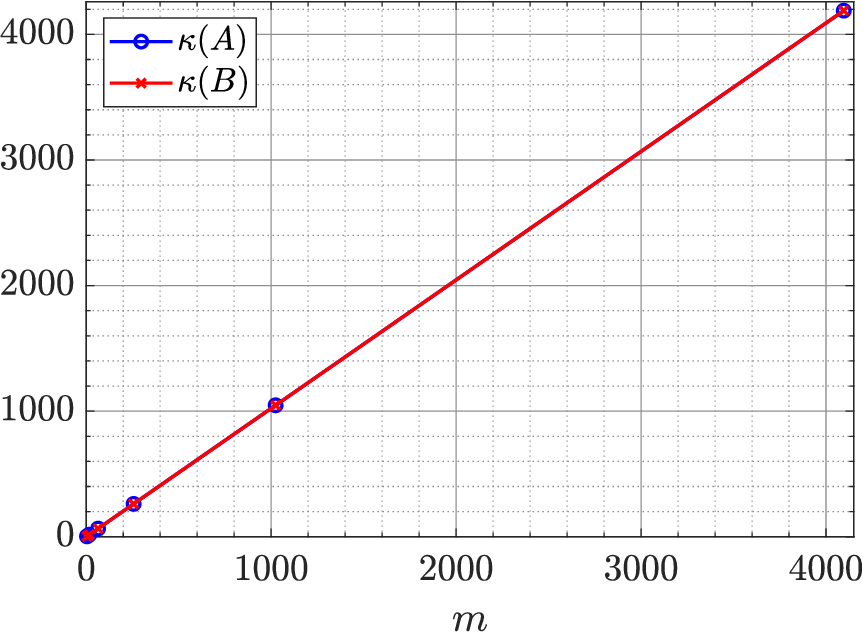}

}
\caption{$2$-norm condition numbers of $A \in \R^{m,m}$ and $B \in \R^{m/2, 
m/2}$ from~\eqref{eq:A} and~\eqref{eq:B} for the Cantor dust with $q = 1/3$ as
functions of $m = 4^k$ for $k = 1, 2, \ldots, 7$.}
\label{fig:condd}
\end{figure}

\begin{figure}
{\centering
\includegraphics[width=0.475\linewidth]{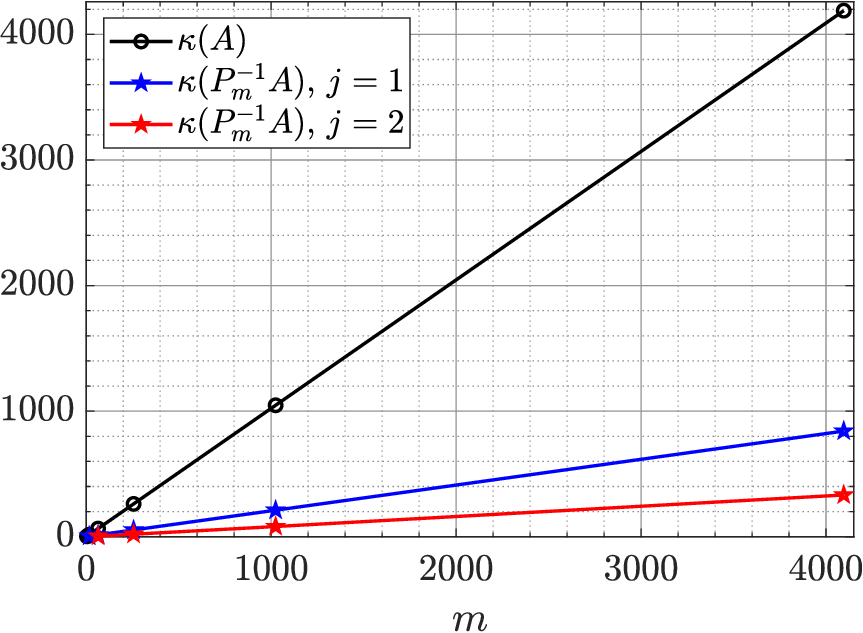}
\hfill
\includegraphics[width=0.475\linewidth]{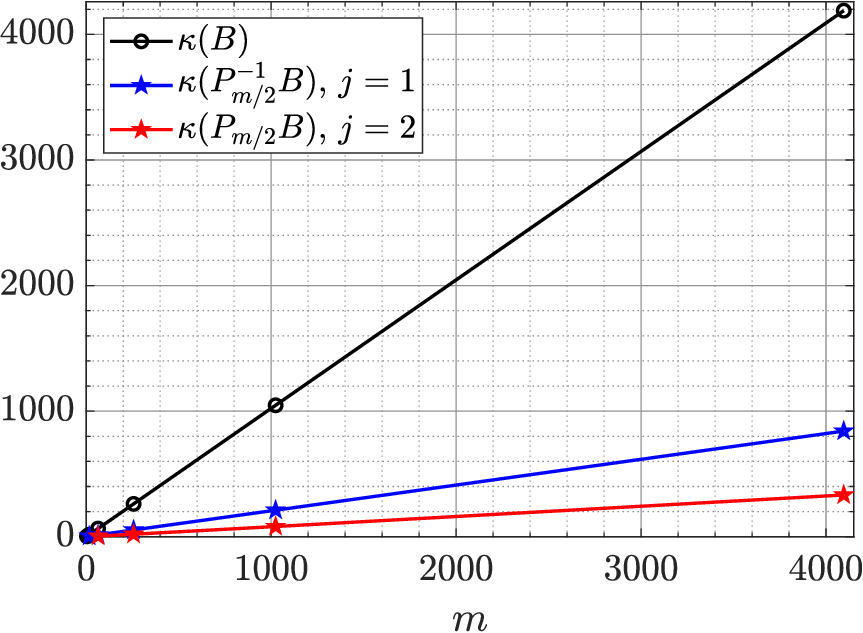}
}
\caption{$2$-norm condition numbers of $P_m^{-1} A \in \R^{m,m}$ and 
$P_{m/2}^{-1} B \in \R^{m/2, m/2}$ as functions of $m = 4^k$ for $q=1/3$, 
$j=1,2$, and $k=j+1,j+2, \ldots,6$.}
\label{fig:ABd-cond}
\end{figure}

\begin{figure}
{\centering
\includegraphics[width=0.475\linewidth]{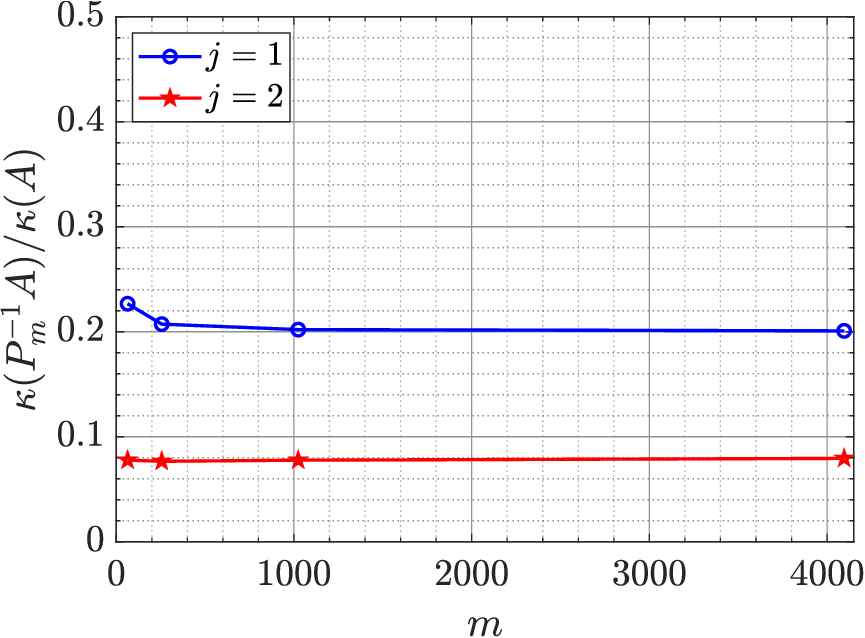}
\hfill
\includegraphics[width=0.475\linewidth]{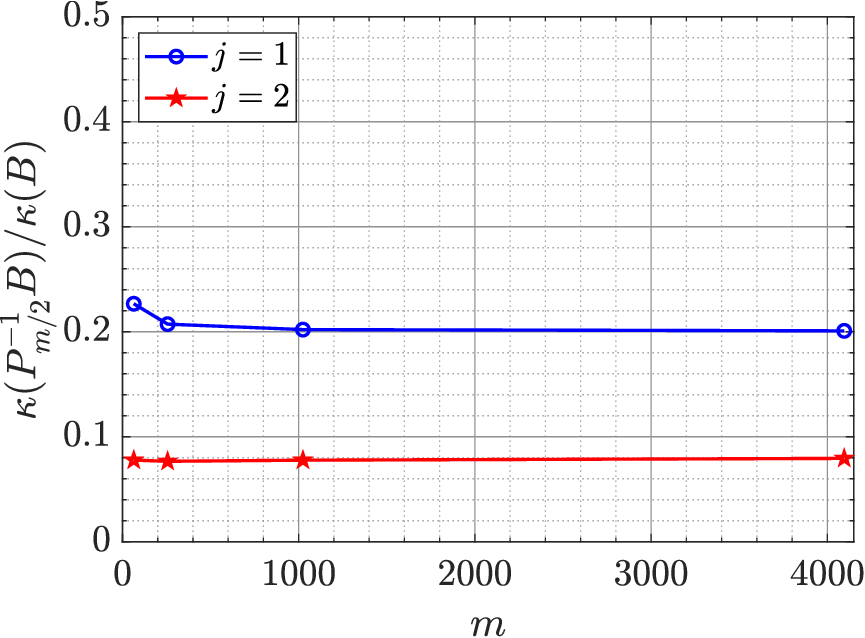}
}
\caption{$\kappa(P_m^{-1}A)/\kappa(A)$ and $\kappa(P_{m/2}^{-1} B)/\kappa(B)$
as functions of $m = 4^k$ for $q=1/3$, $j=1,2$, and $k=3,4,5,6$.}
\label{fig:ABd-cond-2}
\end{figure}

In~\cite[Corollary~3]{minda} it is shown that
\begin{equation}\label{eq:LB-UB}
(1-2q)\sqrt[3]{q} \le \capp(F) \le \sqrt{2} \sqrt[3]{q}.
\end{equation}
For $q=1/3$ the values of the lower and upper bounds rounded to four
significant digits are $0.2311$ and $0.9806$, respectively.
A comparison with~\eqref{eq:LbUb-RR} below illustrates that
we do not expect the bounds~\eqref{eq:LB-UB} to be very tight.
We therefore do not further consider them in our numerical examples.

\begin{example}[Classical Cantor middle third dust]\label{ex:third-d}
The classical Cantor middle third dust is obtained for $q=1/3$.
This example has been considered by Ransford and Rostand
in~\cite[\S5.4]{Ransford2010} where it was proved that
\begin{equation}\label{eq:LbUb-RR}
\capp(F)\in[0.573550,0.575095],
\end{equation}
and their ``best guess'' is
\begin{equation}\label{eq:cap-RR}
\capp(F)\approx \underline{0.57434}50704.
\end{equation}

As for the generalized Cantor sets, we compare our new method with the BIE 
method from~\cite{LSN}.  In that method we first approximate $\capp(D_k)$ using 
$n=2^6$ discretization points on each boundary component. This leads to linear 
algebraic systems of size $nm\times nm$, where $m=4^k$; see 
Remark~\ref{rmk:savings}.
The computed results and the timings for $k=1,2,\ldots,6$ are given in 
Table~\ref{tab:ccd}.  Similarly to Example~\ref{ex:third} we extrapolate the 
values in Table~\ref{tab:ccd} in order to obtain an approximation of $\capp(F)$.
For the approximate values of $\capp(D_k)$ obtained by the BIE method, the 
differences
\begin{equation*}
d_k= \capp(D_{k})-\capp(D_{k+1}),\quad k=1,\dots,5,
\end{equation*}
decrease linearly on a logarithmic scale.  We store these $5$ values in the
vector $d$, and use the MATLAB command {\tt p=polyfit(1:5,log(d),1))} to
compute a linear polynomial $p(x)=p_1x+p_2$ of best approximation in the least 
squares sense for the values $\log(d_k)$.  The computed coefficients are
\begin{equation*}
p_1=-0.981567064346955,\quad p_2=-4.487710197832711.
\end{equation*}
Starting with our computed approximation of $\capp(D_{6})$ we can then
approximate
$\capp(D_k)$ for $k\geq 6$ by extrapolation, i.e.,
\begin{equation*}
\capp(D_k)\approx \capp(D_{6})-\sum_{j=6}^{k-1}\exp(p(j)),\quad k\geq 6.
\end{equation*}
We have $\exp(p(35))<10^{-16}$, and hence we use this formula with $k=35$ for
our estimate
\begin{equation}\label{eq:cap-BIE}
\capp(F)\approx \underline{0.57434}7200461138.
\end{equation}

\begin{table}[t]
\caption{The computed approximate values of $\capp(D_k)$ for $q = 1/3$;
see Example~\ref{ex:third-d}.} \label{tab:ccd}
\begin{center}
\begin{tabular}{l|r|l|r|l|r|r}
\hline
  &   & \multicolumn{2}{c|}{BIE Method~\cite{LSN}}  &
\multicolumn{3}{c}{New Method} \\
\cline{3-7}
$k$ & $m=4^k$  & $\capp(D_k)$ & time (s) & $\capp(D_k)$ & time (s) & iter \\
\hline
 1  & 4       & 0.624190165168316 & 0.31     & 0.560597610169143 & 0.02 & 1\\
 2  & 16      & 0.592984584624405 & 2.27     & 0.569592176189256 & 0.02 & 3\\
 3  & 64      & 0.581332069392025 & 13.60    & 0.572519249447232 & 0.03 & 6 \\
 4  & 256     & 0.576967822959311 & 164.36   & 0.573655973612409 & 0.07 & 12 \\
 5  & 1024    & 0.575330169313215 & 1916.67  & 0.574085928829936 & 0.18 & 25 \\
 6  & 4096    & 0.574715154920085 & 31052.23     & 0.574247683774223  & 0.42 & 
36 \\
 7  & 16384     &  &      & 0.574308476438467 & 1.47 & 54 \\
 8  & 65536     &  &      & 0.574331319690862 & 6.24 & 80 \\
 9  & 262144    &  &      & 0.574339902832891 & 32.13 & 117 \\
 10 & 1048576   &  &      & 0.574343127837383 & 202.43 & 184 \\
\hline
\end{tabular}
\end{center}
\end{table}

Next, we use the new method based on the CSM to compute approximations of
$\capp(D_k)$. Since the size of the preconditioner matrix is $16\times16$, the
method is used without preconditioning technique for $k<3$.  The method with
preconditioning technique is used for $k\ge3$, so that the size of the linear
system is $m\times m$ with $m\ge 64$. Table~\ref{tab:ccd} gives the 
computed approximations of $\capp(D_k)$, the timings, and the 
number of GMRES iteration steps for $k=1,2\dots,10$. We observe that the
computed approximations of $\capp(D_k)$ are increasing, although
$\capp(D_{k+1})\le \capp(D_k)$ (since $D_{k+1}\subset D_k$ for $k\ge 1$).
This is most likely due to the fact that we use the CSM with only one point 
in the interior of each disk $D_{k,j}$, $j=1,2,\ldots,4^k$, so that the  
numerical approximations are not accurate enough to yield a decreasing
sequence. As shown in Table~\ref{tab:ccd}, the BIE method indeed gives
a decreasing sequence of approximations. But here we use $n=2^6$ discretization 
points for each circle $\partial D_{k,j}$ for $j=1,2,\ldots,4^k$, which results
in significantly longer computation times. 

For the approximate values of $\capp(D_k)$ obtained with the CSM in 
Table~\ref{tab:ccd}, the differences
\begin{equation*}
d_k= \capp(D_{k+1})-\capp(D_k),\quad k=1,2,\dots,9,
\end{equation*}
decrease linearly on a logarithmic scale. We store these $9$ values in the
vector $d$ and use the MATLAB command {\tt p=polyfit(1:9,log(d),1))}
to compute a linear polynomial
\begin{equation*}
p(x)=-0.983339218806568\,x-3.806740804822764
\end{equation*}
of best approximation in the least squares sense for the values  $\log(d_k)$.
Starting with our computed approximation of $\capp(D_{10})$ we approximate
$\capp(D_k)$ for $k\geq 11$ by extrapolation, i.e.,
\begin{equation*}
\capp(D_k)\approx \capp(D_{10})+\sum_{j=10}^{k-1}\exp(p(j)),\quad k\geq 11.
\end{equation*}
We have $\exp(p(34))<10^{-16}$, and hence we use this formula with $k=34$ for
our estimate
\begin{equation}\label{eq:cap-CSM}
\capp(F)\approx \underline{0.57434}5031687538.
\end{equation}
This estimate agrees in its first eight significant digits with the estimate 
in~\eqref{eq:cap-RR} obtained in~\cite{Ransford2010}, and all three estimates 
stated above agree in their first five significant digits.

Finally, for $q=1/3$, it is worth mentioning that the new method works for any 
$r_k \in [q^k/\sqrt{2},q^k)$ (instead of~\eqref{eq:rk-d}), i.e., the circles 
can be chosen as large as possible so that each circle $\partial D_{k,j}$ 
encloses the square $S_{k,j}$, and that these circles are disjoint.  Numerical 
experiments (not presented in this paper) show that the new method produces a 
decreasing sequence of approximate values for $\capp(D_k)$ if we choose, for 
example, $r_k=1.25q^k/\sqrt{2}$.  Note that the proof of 
Theorem~\ref{thm:cantor_dust} can be adapted to see that $\capp(F) = \lim_{k 
\to \infty} \capp(D_k)$ also holds for these larger radii.
\end{example}

\begin{example}[Generalized Cantor dust]\label{ex:generalized-d}
We consider now the numerical approximation of $\capp(F)$ for general
$q\in(0,0.5)$.  For the limiting cases $F = \{0,1,1+\i,\i\}$ (corresponding to
$q=0$) and $F = [0,1]\times[0,1]$ (corresponding to $q=0.5$), the capacities
are
\[
\capp(F)=0 \quad \text{and} \quad
\capp(F) = \frac{\Gamma(1/4)^2}{4\pi\sqrt{\pi}}\approx 0.590170299508048,
\]
respectively; see~\cite[Table~1]{LSN}. The new method can be used 
for $q<\sqrt{2}-1$, so that the disks $D_{k,j}$ with the radius $r_k$ given 
by~\eqref{eq:rk-d} are disjoint. Here we use the method to approximate 
the value of $\capp(F)$ for $q=1/20,2/20,\ldots,8/20$.

Similar to Example~\ref{ex:third-d}, the sequences of the computed 
approximations of
$\capp(D_k)$ are not necessarily decreasing.  For this example, the method
generates a decreasing sequence of approximate values for $q=1/20,\ldots,5/20$,
and an increasing sequence for $q=6/20,7/20,8/20$.  Then, using the
approach described in Example~\ref{ex:third-d}, the obtained approximate values
are extrapolated to obtain approximations for $\capp(F)$, which are
stated in Table~\ref{tab:q-d}. 

\begin{table}[t]
\caption{The computed approximations of $\capp(F)$; see 
Example~\ref{ex:generalized-d}.}\label{tab:q-d}
\begin{center}
\begin{tabular}{l|l}
\hline
$q$       & $\capp(F)$ \\
\hline
 $1/20$  & $0.393193419290132$ \\
 $2/20$  & $0.471075541819326$ \\
 $3/20$  & $0.513383693856075$ \\
 $4/20$  & $0.539195036874426$ \\
 $5/20$  & $0.55611125682008$ \\
 $6/20$  & $0.568071614755641$ \\
 $7/20$  & $0.577089193675801$ \\
 $8/20$  & $0.583884621972929$ \\
\hline
\end{tabular}
\end{center}
\end{table}

Similarly to the generalized Cantor sets, it would be of interest to closely 
approximate the values of $\capp(F)$ with a function of $q \in \cc{0, 0.5}$.
After some attempts we came up with
\begin{equation} \label{eqn:approx_cappF}
f(q) = \sqrt{2} \frac{\Gamma(1/4)^2}{4 \pi \sqrt{\pi}} (q(1-q))^{1/4},
\end{equation}
which is shown in Figure~\ref{fig:approx_cappF}.
It is another open question to determine whether there is an exact analytic 
relation between the logarithmic capacities of the generalized Cantor set and 
generalized Cantor dust.
The only relation we are aware of is $\capp(F) \geq 2 \capp(E)$; 
see~\cite[p.~1516]{RanRos07}.

\begin{figure}
{\centering
\includegraphics[width=0.45\linewidth]{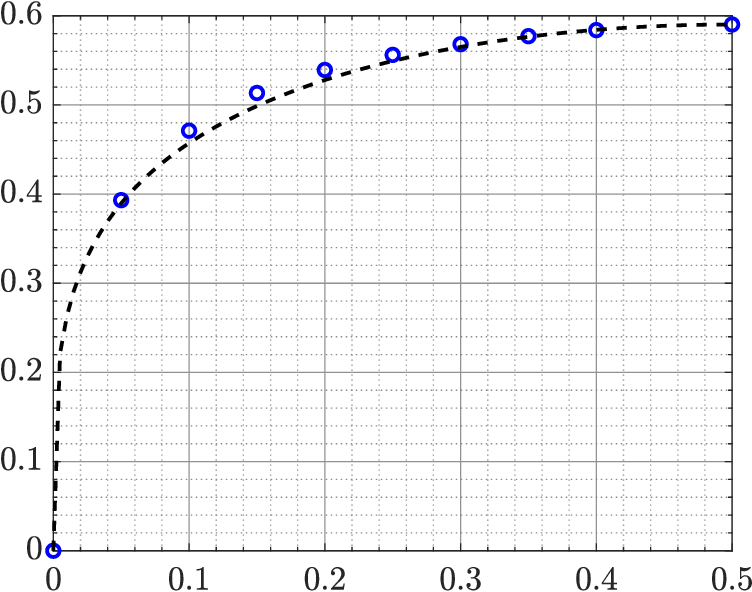}

}
\caption{The computed logarithmic capacity $\capp(F)$ (blue circles) and the 
function $f$ in~\eqref{eqn:approx_cappF} (dashed).}
\label{fig:approx_cappF}
\end{figure}

\end{example}

\section{Concluding remarks}\label{sec:Concl}

In this paper we have applied the CSM to the computation of the logarithmic 
capacity of compact sets consisting of very many ``small'' components.  This 
application allows to use just a single charge point for each component, which 
leads to a significantly more efficient computational method in comparison with 
methods that use discretizations of the boundaries of the different components. 
We have obtained an additional speedup of the method by exploiting the 
structure of the system matrices, and by using a problem-adapted preconditioner 
for the linear algebraic systems.  In the numerical examples we have seen that 
for the same number of components, the new method is faster by a factor of 100 
(sometimes even 1000) than our previous BIE method~\cite{LSN}, while 
maintaining the same high level of accuracy. 
We have applied the method to generalized Cantor sets and the Cantor dust.  We 
are not aware of any other computed approximations of the logarithmic capacity 
of the Cantor dust for $q \neq 1/3$ in the literature.

\bibliographystyle{siam}
\bibliography{capacity}

\end{document}